\documentclass{article}
\usepackage[letterpaper, left=1in, right=1in, top=1in, bottom=1in]{geometry}

\PassOptionsToPackage{hypertexnames=false}{hyperref}  
\usepackage{parskip}

\usepackage[colorlinks=true, linkcolor=blue!70!black, citecolor=blue!70!black,urlcolor=black,breaklinks=true]{hyperref}
\usepackage{microtype}
\usepackage{hhline}

\makeatletter
\newcommand{\neutralize}[1]{\expandafter\let\csname c@#1\endcsname\count@}
\makeatother

\usepackage{algorithm}

\usepackage{natbib}
\bibpunct{(}{)}{;}{a}{,}{,}

\usepackage{amsthm}
\usepackage{mathtools}
\usepackage{amsmath}
\usepackage{bbm}
\usepackage{amsfonts}
\usepackage{amssymb}

\usepackage{xpatch}

\newtheorem{theorem}{Theorem}
\newtheorem{lemma}{Lemma}
\newtheorem{proposition}{Proposition}
\newtheorem{assumption}{Assumption}
\newtheorem{corollary}{Corollary}
\newtheorem{conjecture}{Conjecture}
\newtheorem{definition}{Definition}
\theoremstyle{definition}
\newtheorem{example}{Example}

\theoremstyle{remark}
\newtheorem{remark}{Remark}
\AtBeginEnvironment{remark}{%
  \pushQED{\qed}%
}
\AtEndEnvironment{remark}{\popQED\endexample}

\makeatletter
  \renewenvironment{proof}[1][Proof]%
  {%
   \par\noindent{\bfseries\upshape {#1.}\ }%
  }%
  {\qed\newline}
  \makeatother

\xpatchcmd{\proof}{\itshape}{\normalfont\proofnameformat}{}{}
\newcommand{\proofnameformat}{\bfseries}

\usepackage[nameinlink,capitalize]{cleveref}

\crefformat{equation}{#2(#1)#3}
\Crefformat{equation}{#2(#1)#3}

\Crefformat{figure}{#2Figure #1#3}
\Crefname{assumption}{Assumption}{Assumptions}
\Crefformat{assumption}{#2Assumption #1#3}

\usepackage{crossreftools}
\usepackage{xparse}

\ExplSyntaxOn
\DeclareDocumentCommand{\XDeclarePairedDelimiter}{mm}
 {
  \__egreg_delimiter_clear_keys: 
  \keys_set:nn { egreg/delimiters } { #2 }
  \use:x 
   {
    \exp_not:n {\NewDocumentCommand{#1}{sO{}m} }
     {
      \exp_not:n { \IfBooleanTF{##1} }
       {
        \exp_not:N \egreg_paired_delimiter_expand:nnnn
         { \exp_not:V \l_egreg_delimiter_left_tl }
         { \exp_not:V \l_egreg_delimiter_right_tl }
         { \exp_not:n { ##3 } }
         { \exp_not:V \l_egreg_delimiter_subscript_tl }
       }
       {
        \exp_not:N \egreg_paired_delimiter_fixed:nnnnn 
         { \exp_not:n { ##2 } }
         { \exp_not:V \l_egreg_delimiter_left_tl }
         { \exp_not:V \l_egreg_delimiter_right_tl }
         { \exp_not:n { ##3 } }
         { \exp_not:V \l_egreg_delimiter_subscript_tl }
       }
     }
   }
 }

\keys_define:nn { egreg/delimiters }
 {
  left      .tl_set:N = \l_egreg_delimiter_left_tl,
  right     .tl_set:N = \l_egreg_delimiter_right_tl,
  subscript .tl_set:N = \l_egreg_delimiter_subscript_tl,
 }

\cs_new_protected:Npn \__egreg_delimiter_clear_keys:
 {
  \keys_set:nn { egreg/delimiters } { left=.,right=.,subscript={} }
 }

\cs_new_protected:Npn \egreg_paired_delimiter_expand:nnnn #1 #2 #3 #4
 {
  \mathopen{}
  \mathclose\c_group_begin_token
   \left#1
   #3
   \group_insert_after:N \c_group_end_token
   \right#2
   \tl_if_empty:nF {#4} { \c_math_subscript_token {#4} }
 }
\cs_new_protected:Npn \egreg_paired_delimiter_fixed:nnnnn #1 #2 #3 #4 #5
 {
  \mathopen{#1#2}#4\mathclose{#1#3}
  \tl_if_empty:nF {#5} { \c_math_subscript_token {#5} }
 }
\ExplSyntaxOff

\XDeclarePairedDelimiter{\supnorm}{
  left=\lVert,
  right=\rVert,
  subscript=\infty
  }

\usepackage[utf8]{inputenc} 
\usepackage[T1]{fontenc}    
\usepackage{url}            
\usepackage{booktabs}       
\usepackage{amsfonts}       
\usepackage{nicefrac}       
\usepackage{microtype}      
\usepackage{authblk}
\usepackage{tocloft}            

\usepackage{enumitem}

\usepackage{breakcites}

\usepackage{etoolbox}
\usepackage{comment}
\newtoggle{draft}
\togglefalse{draft}

\usepackage{color-edits}

\usepackage{mathrsfs}

\usepackage{algorithm}
\usepackage{verbatim}
\usepackage[noend]{algpseudocode}

\usepackage{multicol}

\usepackage{colortbl}
\usepackage{bbm}
\usepackage{setspace}

\usepackage{transparent}

\usepackage{inconsolata}
\usepackage[scaled=.90]{helvet}
\usepackage{xspace}

\usepackage{pifont}

\usepackage{tikz-cd}

    \newcommand{\lp}{\left(}
    \newcommand{\rp}{\right)}
    
    \newcommand{\Med}{\mathrm{Med}\,}
    \newcommand{\E}{{\mathbb E}}
    
    \newcommand{\R}{{\mathbb R}}
    \renewcommand{\P}{{\mathbb P}}

    \newcommand{\PP}{\mathbb{P}}

    \newcommand{\nm}{\eps_{*}}
    \newcommand{\nms}{\eps_{*}^2}

    \renewcommand{\b}[1]{\boldsymbol{\mathbf{#1}}}
    \renewcommand{\vec}[1]{\b{#1}}

    \newcommand{\Xf}[1]{G_{#1}}

    \newcommand{\eps}{\epsilon}

    \newcommand{\conv}{\mathrm{conv}}
    \newcounter{rcnt}[section]

    \newcommand{\PPf}{\PP^{(n)}}

    \def\argmin{\mathop{\rm argmin}}
    \def\argmax{\mathop{\rm argmax}}
    \newcommand{\Ef}{\mathrm{E}_{f^{*}}}

    \newcommand{\diam}{\operatorname{diam}}

    \newcommand{\Span}{\operatorname{Span}}

    \newcommand{\QQ}{\mathbb Q}
    
    \newcommand{\erm}{\widehat{f}_n}
    \newcommand{\ft}{f^{*}}

    \newcommand{\Lhat}{\mathcal{\widehat{L}}}
    \newcommand{\Gr}{{\operatorname{Gr_{n + 1, n}}}}
    \def\ddefloop#1{\ifx\ddefloop#1\else\ddef{#1}\expandafter\ddefloop\fi}
    \def\ddef#1{\expandafter\def\csname c#1\endcsname{\ensuremath{\mathcal{#1}}}}
    \ddefloop ABCDEFGHIJKLMNOPQRSTUVWXYZ\ddefloop
    \DeclarePairedDelimiter{\abs}{\lvert}{\rvert} %
    
    \DeclarePairedDelimiter{\crl}{\{}{\}}

    \title{On the Variance, Admissibility, and Stability of Empirical Risk Minimization}

\date{}

\begin{document}
\author[1]{ Gil Kur}
\author[2]{Eli Putterman}
\author[1]{Alexander Rakhlin}
\affil[1]{MIT}
\affil[2]{Tel Aviv University}

\maketitle

\begin{abstract}
    It is well known that Empirical Risk Minimization (ERM) may attain minimax suboptimal rates in terms of the mean squared error \citep{birge1993rates}. In this paper, we prove that, under relatively mild assumptions, the suboptimality of ERM \emph{must} be due to its large bias. Namely, the variance error term of ERM is bounded by the minimax rate. In the fixed design setting, we provide an elementary proof of this result using the probabilistic method. Then, we extend our proof to the random design setting for various models. 
    In addition, we provide a simple proof of Chatterjee's admissibility theorem \cite[Theorem 1.4]{chatterjee2014new}, which states that in the fixed design setting, ERM cannot be ruled out as an optimal method, and then we extend this result to the random design setting. We also show that our estimates imply the \textit{stability} of ERM, complementing the main result of \cite{caponnetto2006stability} for non-Donsker classes.  Finally, we highlight the somewhat irregular nature of the loss landscape of ERM in the non-Donsker regime, by showing that functions can be close to ERM, in terms of $L_2$ distance, while still being far from almost-minimizers of the empirical loss.
\end{abstract}
\section{Introduction}\label{submission}

Maximum Likelihood (MLE) and the method of Least Squares (LS) are fundamental procedures in statistics. The study of the asymptotic consistency of MLE has been central to the field for almost a century \citep{wald1949note}. Along with consistency, its failures have been thoroughly investigated throughout the history of statistics \citep{NeymanScott48, bahadur1958examples, Ferguson82}. In the setting of non-parametric estimation, the seminal work of \citep{birge1993rates} provided \textit{sufficient} conditions for minimax optimality (in a non-asymptotic sense) of LS while also presenting an example of a model class where this basic procedure is sub-optimal. Three decades later, we still do not have necessary and sufficient conditions for minimax optimality of LS---equivalently, Empirical Risk Minimization (ERM) with square loss---in general. While the present paper does not resolve this question, it makes several steps towards understanding its behavior in large models.

Beyond intellectual curiosity, the question of the minimax optimality of the LS is driven by the desire to understand the current practice of fitting large or overparametrized models, such as neural networks, to data (cf. \citep*{belkin2019does,bartlett2020benign}). At present, there is little theoretical understanding of whether unregularized data-fitting procedures are optimal, and studying their statistical properties may lead to new methods with improved performance.

In addition to minimax optimality, many other essential properties of LS on large models are yet to be understood. For instance, little is known about its \textit{stability} under data perturbations. It is also unclear whether \textit{approximate} minimizers of empirical loss enjoy similar statistical properties as the exact solution. Conversely, one may ask whether in the landscape of possible solutions, a small perturbation of the minimizer output by LS itself is a near-optimizer of empirical loss.

The contribution of this paper is to provide novel insights into the aforementioned questions for convex classes of functions in a quite generic setting. In detail, we show the following:
\begin{enumerate}
    \item  We prove that in the fixed design setting, also known as the Gaussian sequence model, the variance (error term) of ERM is upper bounded by the minimax rate of estimation. Thus, if  ERM is minimax suboptimal, it must be due to the bias (in terms of the bias-variance decomposition).
    
    \item In the random design, obtaining a similar result becomes much subtler. We use two different approaches to control the variance: 
    \begin{enumerate}
        \item  We derive an upper bound for the variance under a uniform boundedness assumption on the class, via an empirical process approach. This bound also implies that, under classical assumptions in M-estimation, the variance is at most the minimax rate. 
    \item  Under an isoperimetry assumption on the noise, we upper bound the expected conditional variance of ERM. Furthermore, under an additional isoperimetry assumption on the covariates, we upper bound the variance of ERM on any robust learning architecture (namely, a class consisting of functions which are all $O(1)$-Lipschitz) which almost interpolates the observations, cf. \cite{bubeck2023universal}).

    \end{enumerate}

    \item It is known that ERM is always admissible in the fixed design setting \citep*{chatterjee2014new,chen2017note}; that is, for any convex function class, no estimator has a lower error than ERM (up to a multiplicative absolute constant) on \textit{every} regression function. We provide a short proof of this result via a fixed-point theorem. Using a similar approach, we also prove a somewhat weaker result in the random design case, generalizing the main result of \cite{chatterjee2014new}. 

    \item We show that ERM in the fixed design setting is stable, in the sense that all almost-minimizers (up to the minimax rate) of the squared loss are close in the space of functions. In the random design setting, we prove a non-asymptotic analogue of the asymptotic analysis in \cite{caponnetto2006stability}, and extends its scope to non-Donsker classes.

    \item The last item implies that for non-Donsker classes, any almost-minimizer of the squared loss is close to the minimizer with respect to the underlying population distribution. Our final result shows that the converse is incorrect. We prove that for any non-Donsker class of functions, there exists a target regression function such that, with high probability, there exists a function with high empirical error near the ERM solution. This means that the landscape of near-solutions is, in some sense, irregular. 
\end{enumerate}

\subsection*{Conclusions} 
Our results show that the variance error term of ERM is at most the minimax rate in two distinct regimes. First, in the classical regime \citep*{van2000empirical}, where the function class is fixed and the number of samples is increasing. And secondly,  the  ``benign overfitting'' setting \citep*{belkin2019does,bartlett2020benign}, in which the ``capacity'' of the class is large compared to the number of samples. In both settings, our work implies that the minimax optimality of ERM depends only on its bias term. For models with ``few'' parameters, \emph{computationally efficient} bias correction methods do exist and are commonly used in practice (cf. \citep*{efron1994introduction}). However, these methods may fail in large function classes due to bias, leading to statistical sub-optimality. Our work reveals the importance of developing computationally efficient debiasing methods for rich function classes, including nonparametric and high-dimensional models. Our main message is that the occasional poor performance of ERM in practice can be explained by its large bias, thereby motivating the development of new debiasing procedures. If such methods exist, they may significantly improve the statistical performance of ERM.

\subsection{Prior Work}
\paragraph{Stability of ERM} The stability of learning procedures \cite{bousquet2002stability}, which was an active area of research in the early 2000's, has recently seen a resurgence of interest because of its connections to differential privacy and to robustness of learning methods with respect to adversarial perturbations. In the interest of space, we only compare present results to those of \cite{caponnetto2006stability}. In the latter paper, the authors showed that the $L_1$-diameter of the set of almost-minimizers of empirical error (with respect to any loss function) asymptotically shrinks to zero as long as the perturbation is $o(n^{-1/2})$ and the function class is Donsker. The analysis there relies on passing from the empirical process to its associated Gaussian process in the limit and on studying the uniqueness of its maximum using anti-concentration properties. While the result there holds without assuming that the class is convex, it is limited by (a) its asymptotic nature and (b) the assumption that the class is not too complex. In contrast, the present paper uses more refined non-asymptotic concentration results, at the expense of additional assumptions such as convexity and minimax optimal lower and upper isometry remainders. Crucially, the present result, unlike that of \cite{caponnetto2006stability}, holds for non-Donsker classes---those for which the empirical process does not converge to the Gaussian process.

\paragraph{Shape-constrained regression} The term ``shape-constrained regression'' refers to function classes consisting of functions with a certain ``shape'' property, such as convexity or monotonicity \citep{samworth2018special}. In these problems, a common theme is that the statistical behavior of the class undergoes a phase transition when the domain dimension $d \geq 1$ reaches a certain value. For instance, in convex (Lipschitz) regression, the ERM procedure is \textbf{only} minimax\footnote{In other shape-constrained models, however, the ERM is minimax optimal even in high dimensions, such as isotonic regression and log-concave density estimation \citep*{han2019isotonic,kur2019optimality,carpenter2018near, kur2020convexsub,Kur20}} optimal when $ d \leq 5$ \citep*{seijo2011nonparametric,han2016multivariate,kim2016global,seijo2011nonparametric,guntuboyina2012optimal}. Our results show that ERM's sub-optimality in shape-constrained regression is solely due to its high bias. These results also align with the empirical observation that for the problem of estimation of convex sets, the ERM has a bias towards ``smooth" convex sets \citep*{soh2019fitting, ghosh2021max}.

\paragraph{High-dimensional statistics} In classical statistics, the MLE typically has a low bias compared to its variance, and the standard approach is to introduce bias into the procedure to reduce the variance, overall achieving a better trade-off, see \cite[\S 1]{sur2019modern} and references therein. In contrast, in high-dimensional models, the MLE may suffer from high bias even in tasks such as logistic regression and sparse linear regression, \cite{candes2020phase,javanmard2018debiasing}. Our results align with this line of work, showing that high bias may also arise in regression over rich function classes.
\subsection*{Organization}
In \S \ref{S:MR},  we present our model and all our results.  In \S \ref{ss:ronT}, we discuss the optimality of our bounds in the random design setting. In \S \ref{S:SP}, we provide sketches of some of our proofs to give some flavor to our techniques. Finally, in \S \ref{S:pfs}, we provide the proofs for all our results.
\section{Main Results}

\label{S:MR}
\subsection{Preliminaries}\label{sec:prelims}
Let $\mathcal X$ be some fixed domain, $\cF$ be a class of functions from $\cX$ to $\R$, and $\ft \in \cF$ an unknown target regression function. We are given  $n$ data points $X_1 \ldots,X_n \in \mathcal X$ and $n$ noisy observations
\begin{align}
\label{eq:data_model}
    Y_i = \ft(X_i) + \xi_i, \quad i = 1, \ldots, n
\end{align}
which we denote by $\cD:= \{(X_i,Y_i)\}_{i=1}^{n}$, and $\vec \xi := (\xi_1,\ldots,\xi_n)$ for the the random noise vector.

In the \textbf{fixed design} setting, the observations  $X_1=x_1,\ldots,X_n=x_n$ are arbitrary and fixed, and we denote the uniform measure on this set of points by $\PPf$. 

In the \textbf{random design} setting, the data points $\vec X:=(X_1,\ldots,X_n)$ are drawn i.i.d. from a probability distribution over $\cX$, denoted by $\PP$, and the noise vector $\vec \xi$ is drawn independently of $\vec X$. Note that this model is general enough to cover the high-dimensional setting, as both the function class $\cF$ and the distributions of $\vec \xi$ and $\vec X$ are allowed to depend on the number of samples $n$.

An estimator for the regression task is defined as a measurable function $ \bar{f}_n : \cD \mapsto \{\cX\to\mathbb{R}\}$, that for any realization of the input $\cD$, outputs some real-valued measurable function on $\cX$. The risk of $\bar{f}_n$ is defined as
    \begin{equation}\label{eq:risk_def}
        \cR(\bar{f}_n,\cF,\QQ):= \sup_{\ft \in \cF} \E_{\cD} \int (\bar{f}_n - \ft)^2 d\QQ,
    \end{equation}
where $\QQ =\PP$ in the random design case, and $\QQ=\PPf$ in the fixed design case. Note that in fixed design, the expectation $\E_{\cD}$ is taken over the noise $\vec\xi$, while in random design the expectation $\E_{\cD}$  is taken both over the random data points $\vec X$ and noise $\vec \xi$.  The minimax rate is defined via 
\begin{align}
\label{eq:minimax_rate_def}
    \cM(n,\cF,\QQ):= \inf_{\bar{f}_n}\cR(\bar{f}_n,\cF,\QQ).
\end{align}
In the fixed design setting, we also denote the minimax rate by $\cM(\cF,\PPf)$, as the dependence in $n$ is already present in $\PPf$.  

The most natural estimation procedure is the LS or ERM with squared loss, defined as
\begin{equation}\label{Eq:LSE}
\erm \in \argmin_{f \in \mathcal F} \sum_{i = 1}^n (f(X_i) - Y_i)^2.
\end{equation}
\noindent{When studying fixed design, we will abuse notation and treat $\erm$ as a vector in $\mathbb R^n$. We emphasize that many of our results hold for many other estimators, including various \emph{regularized} ERM procedures (see the relevant remarks below).  In both the fixed and random design settings, we shall assume the following:} 
\begin{assumption}\label{A:Convex}
$\cF$ is a closed convex subset of $L_2(\QQ)$, where $\QQ \in \{\PPf,\PP\}$.
\end{assumption}
The convexity of $\cF$ means that any $f,g \in \cF$ and $\lambda \in [0,1]$: $\lambda f + (1-\lambda)g \in \cF$; closedness means that for any sequence $\{f_n\}_{n = 1}^\infty \subset \cF$ converging to $f$ with respect to the norm of $L_2(\QQ)$, the limit $f$ lies in $\cF$. The closedness ensures that $\erm$ is well-defined.

Assumption~\ref{A:Convex} is standard in studying the statistical performance of the ERM (cf. \cite*{lee1996importance,bartlett2005local,mendelson2014learning}).
In particular, under this assumption, the values of $\erm$ at the observation points $X_1, \ldots, X_n$ are uniquely determined for any $\vec{\xi}$. Note that, in the random design case, the values of a function $f \in \mathcal F$ at the points $X_1, \ldots, X_n$  may not uniquely identify $f$ in the class.

In addition to $\erm$, we analyze properties of the set of $\delta$-approximate minimizers of empirical loss, defined for $\delta>0$ via
\begin{equation}\label{Eq:almost}
\cO_{\delta}:= \left\{f \in \cF: \frac{1}{n}\sum_{i=1}^{n}(Y_i - f(X_i))^2  \leq  \frac{1}{n}\sum_{i=1}^{n}(Y_i - \erm(X_i))^2 + \delta \right\}.
\end{equation}
Note that $\cO_{\delta}$ is a random set, in both fixed and random designs. 

It is well-known that the squared error of an estimator, $\bar{f}_n$, in particular that of LS, decomposes into variance and bias components (respctively):
\begin{equation}\label{eq:biasvariance_decomp}
\begin{aligned}
    \E_{\cD} \int (\erm - \ft)^2d\QQ &= \E_{\cD} \int (\erm - \E_{\cD}\erm )^2d\QQ+  \int (\E_{\cD} \erm - \ft)^2d\QQ :=V(\erm) + B^2(\erm),
\end{aligned}
\end{equation}
where $\QQ = \PPf$ in the fixed design setting and $\QQ=\PP$ in the random design setting. Also, for simplicity of the presentation of our results, we denote the maximal variance error term of $\bar{f}_n$ by  $ \cV(\erm,\cF,\QQ)$, i.e. 
\[
    \cV(\erm,\cF,\QQ) := \sup_{\ft \in \cF}V(\erm).
\]
In the random design setting, we also have the law of total variance:
\begin{equation}\label{Eq:Eve}
\begin{aligned}
    V(\erm) &=   \E_{\vec X}\E_{\vec \xi}\left[\int \left(\erm - \E_{\vec \xi}\left[\erm |\vec X \right]\right)^2\,d\PP \right]+\E_{\vec X}\left[\int\left(\E_{\vec \xi}\left[\erm |\vec X \right] - \E_{\vec X,\vec \xi}\left[\erm\right] \right)^2d\PP\right]
    \\&:= \E_{\vec X} V_{\vec \xi}(\erm|\vec X) + V(\E_{\vec \xi}(\erm|\vec X)),
\end{aligned}
\end{equation}
we refer to the two terms as the expected conditional variance and the variance of the conditional expectation, respectively. 
We conclude this introductory section with a bit of notation and a definition.

\paragraph{Notation} We use the notation of $\asymp,\gtrsim,\lesssim$ to denote equality/inequality up to an absolute constant. We use $\| \cdot \|=\|\cdot\|_{L_2(\PP)}$ to denote the $L_2(\PP)$ norm, and $\| \cdot \|_{n}$ to denote the $L_2(\PP^{(n)})$ norm (that is equal to the Euclidean norm scaled by $1/\sqrt{n}$). Finally, given a function $f: S \to T$ between metric spaces, we define its Lipschitz constant as $\| f \|_{\rm Lip} = \sup_{a, b \in T, a \neq b} \frac{d_T(f(a), f(b))}{d_S(a, b)},$ and we say ``$f$ is $L$-Lipschitz'' when its Lipschitz constant is at most $L$. Finally, we denote by $\diam_{\QQ}(\cH)$ the $L_2(\QQ)$ diameter of a set of functions $\cH$.
\begin{definition}\label{Def:Covering}
 Let $\eps \geq 0$, $\cF\subseteq\{\mathcal{X}\to\mathbb{R}\}$ and $d(\cdot,\cdot)$ a pseudo-metric on $\cF$. We call a set $S \subset \cF$ an $\epsilon$-net of $\cF$ with respect to $d$ if for any $f \in \cF$ there exists $g \in S$ with $d(f, g) \le \epsilon$. We denote by $\cN(\eps,\cF,d)$ the $\epsilon$-covering number of $\cF$ with respect to $d$, that is, the minimal positive integer $N$ such that $\cF$ admits an $\epsilon$-net of cardinality $N$.
\end{definition}



\subsection{Fixed design setting}
\label{sec:fixed}
In this part, we consider some fixed $(\cF,\PPf)$ and assume the following: 
\begin{assumption}\label{A:normal}
The noise vector $\vec \xi$ is distributed as an isotropic Gaussian, i.e. $\vec\xi \sim N(0,I_{n \times n})$.
\end{assumption}
Also, for every $f \in \cF$ and $r \geq 0$, we denote by
\[
    B_n(f,r):= \{ g \in \cF: \|g - f\|_{n} \leq r\}.
\]
Our first result provides an exact characterization of the variance (up to a multiplicative absolute constant) under Assumptions \ref{A:Convex}-\ref{A:normal}. In order to state it, for a fixed $\ft \in \cF$, we define the following set:
\begin{equation}\label{Eq:cH}
    \cH_{*}:= B_n(\E_{\cD}\erm,2\sqrt{V(\erm)}).   
\end{equation}
In words, when the underlying function $\ft$ is fixed, we consider the ERM as a random vector (depending on the noise), whose expectation we denote by $\E_{\cD} \erm$. $\cH_*$ is then just a neighborhood around the expected ERM with a radius of order the square root of the variance error term of $\erm$, when the underlying function is $\ft \in \cF$.  

We can now state our first result, which uses the notion of the set $\cO_\delta$ of $\delta$-approximate minimizers from \eqref{Eq:almost}.
\begin{theorem}\label{T:Fix}
Under Assumptions \ref{A:Convex}-\ref{A:normal}, for any $\ft \in \cF$, the following holds:
\[
V(\erm) \asymp \cM(\cH_{*},\PPf),
\]
and in particular $\cV(\erm,\cF,\PPf) \lesssim \cM(\cF,\PPf)$.  Furthermore, for $\delta:= \delta(\ft,n) \lesssim \cM(\cH_{*},\PPf)$, the event
\begin{equation}\label{Eq:Stability}
     \sup_{f \in \cO_{\delta}} \int (f - \E_{\cD} \erm)^2d\PPf \asymp  \cM(\cH_{*},\PPf)
\end{equation}
holds with probability  at least $\max\{1 - 2\exp(-cn \cdot \cM(\cH_{*},\PPf)),0.9\}$, where $c \in (0,1)$ is an absolute constant.
\end{theorem}
\noindent{Theorem~\ref{T:Fix} establishes our first claim: the variance of ERM is bounded above (up to a multiplicative absolute constant) by the minimax rate of estimation on $\cH_*$. Since $\cH_*$ is contained in $\mathcal F$, the variance of ERM is upper bounded by $ \cM(\cF,\PPf)$. \textbf{This implies that if ERM is minimax sub-optimal, then it must be due to its bias}. The theorem also incorporates a stability result: not only is the  ERM close to its expected value $\E_{\cD} \erm$ with high probability, but any approximate minimizer (up to an excess error of $\delta^2$) is close to $\E_{\cD} \erm$ as well.}
Our next proposition complements Theorem \ref{T:Fix} above, providing a lower bound on $\cV(\erm,\cF,\PPf)$:
\begin{proposition}\label{T:FixLower}
Under Assumptions \ref{A:Convex}-\ref{A:normal}, the following holds:
\[\sqrt{\cV(\erm,\cF,\PPf)} \gtrsim \cM(\cF,\PPf).\]
\end{proposition}
\noindent{Note that there is a multiplicative gap of order $\cM(\cF,\PPf)$ between of Theorem \ref{T:Fix}  and Proposition~\ref{T:FixLower}. We leave it as an open problem whether the bound of Proposition~\ref{T:FixLower} can be improved under these general assumptions.}

Next, we state the \textit{admissibility} theorem of ERM,  established by \cite{chatterjee2014new}.  
\begin{theorem}[Chatterjee's Admissibility Theorem]\label{C:A}
     Let $\bar{f}_n:\cD \to \R^n$ be some estimator. Then, under Assumptions \ref{A:Convex}-\ref{A:normal}, there exists an underlying function $\ft \in \cF$ (that depends on $\bar{f}_n$) such that 
     \begin{equation}\label{Eq:Admis}
              \E_{\cD}\int (\erm - \ft)^2d\PP^{(n)}\leq C \cdot \E_{\cD}\int (\bar{f}_n - \ft)^2d\PP^{(n)},
     \end{equation}
     where $C \geq 1$ is a universal constant that is independent of $\bar{f}_n$, $\cF$, $\PPf$. 
\end{theorem}
In words, this result states that for any estimator $\bar{f}_n$, there \textit{exists} a target function $\ft\in\cF$ such that ERM over the data drawn according to \eqref{eq:data_model} has error which is no worse (up to an absolute constant) than that of $\bar{f}_n$. Hence, while ERM may be suboptimal for some models $\ft\in\cF$, it cannot be ruled out completely as a learning procedure. 

The original proof is highly non-trivial and quite complicated. \cite{chen2017note} provided a bit simpler proof, with a better estimate of $1.65 \cdot 10^5$ for the constant $C > 1$ in \eqref{Eq:Admis}. In this work, we provide a new \textbf{approach that offers a simplified perspective} on this profound theorem. In addition, it yields a much better bound of $10^2$ rather than $1.65 \cdot 10^5$; as we have not attempted to optimize the constants, we believe this can be improved further. 
We show that admissibility hinges on the existence of a target regression function $\ft \in \cF$ such that the estimator not only has a ``small bias'' but is also ``stable'' around it. Remarkably, under compactness of $\cF$, the existence of such a target function is ensured by a purely topological argument—Brouwer's fixed-point theorem. From a statistical perspective, this has a simple interpretation: a ``stable'' estimator cannot have a ``large'' bias on every target function within a compact function class.



Also, it is worth noting that if the class is centrally symmetric, i.e., $\cF = -\cF$, then $\ft \equiv 0$, as we know that $B^2(\erm) = 0$, and therefore, the zero function is estimated in error that is at most of the minimax rate. This result may seem surprising, as in the centrally symmetric case, it holds that
\[
  0 \in \argmax_{\ft \in \cF} w_n(B_n(\ft,r)),
\]
where $w_n(\cdot)$ denotes the \emph{Gaussian} complexity---which implies that \emph{local} minimax rate around the origin is maximal (see Lemma \ref{Lem:Minimax} below). We also refer to  \cite*[Pg. 3007]{wei2020gauss} and a  recent paper of \citep{aolaritei2025revisiting}  for further details.
 \subsubsection*{Concluding Remarks}
 \begin{remark}
     The first part of of Theorem \ref{T:Fix} holds for any estimator $\widehat{g}_n$ for which the map $\vec{\xi} \mapsto \widehat{g}_n(\vec{\xi})$ is $O(1)$-Lipschitz, namely, 
     \[
            V(\widehat{g}_n) \asymp \cM(\cH_{*},\PPf).
     \]
 Furthermore, when the distribution of the noise $\vec \xi$ satisfies the Lipschitz Concentration Property (see Assumption \ref{A:CCP} below), our proof implies that
     \[
        V(\erm) \lesssim \cM(\cH_{*},\PPf).
     \]
    Note that the left term is the \textbf{minimax rate under isotropic Gaussian noise} that potentially can be larger than the minimax rate under this noise distribution.
\end{remark}
\begin{remark}\label{Rem:StabilityTh}
One can verify that for $ \delta(\ft,n) \gg \cM(\cH_{*},\PPf)$, Theorem \ref{T:Fix} cannot be true in its full generality. Therefore, the stability threshold of $\delta(\ft,n)$ is tight, up to a multiplicative absolute constant.
\end{remark}
\begin{remark}
         The proof of Proposition \ref{T:FixLower} is specific to $\erm$ and cannot be extended immediately to other estimators. However, it holds for any isotropic noise distribution.
\end{remark}
 \begin{remark}
Under the additional assumption of $\cF$ being a \emph{compact} class, our proof demonstrates that the admissibility property of Theorem~\ref{C:A} is valid for any $\widehat{g}_n$ such that $\vec \xi \mapsto \widehat{g}_n(\vec \xi)$ is $O(1)$-Lipschitz. 
\end{remark}



\subsection{Random design setting}\label{sec:random}
We now turn our attention to the random design setting. Here, we establish similar results to the previous sub-section, albeit under additional assumptions, and with significantly more effort. Unlike the fixed design case, we cannot provide an exact characterization of the variance of ERM. We shall use two different approaches to estimate the variance of the error term. In the first approach, we use classical tools of empirical process theory together with assumptions that are commonly used in M-estimation \citep{van2000empirical}. The second approach, which is inspired by our fixed-design approach, relies heavily on isoperimetry and concentration of measure (cf. \cite{ledoux2001concentration}).

Throughout this part, we assume for simplicity of presentation that the $L_2(\PP)$-diameter of the function class is independent of $n$.
\begin{assumption}\label{A:Diameter}
    There exist absolute constants $C,c > 0$ such that  $ c \leq \diam_{\PP}(\mathcal F) \leq C$. 
\end{assumption}

The classical work of \cite{yang1999information} provides a characterization of the minimax rate  $(n,\cF,\PP)$  under appropriate assumptions (such as normal noise, and uniform boundedness of $\cF$, and richness of $\cF$). They proved that the minimax rate is the square of the solution of the following (asymptotic) equation 
\begin{equation}\label{Eq:minimax}
    \log \cN(\eps,\cF,\PP) \asymp n\eps^2,
\end{equation}
 where $\cN(\eps,\cF,\PP)$ is the $\eps$-covering number of $\cF$ in terms of $L_2(\PP)$ metric (see Def. \ref{Def:Covering} above). We denote this point by $\eps_{*}=\eps_{*}(n)$, and even under less restrictive assumptions, $\nms$ also lower bounds the minimax rate, up to a multiplicative factor of $O(\log n)$. Also, we remark that it is well known that ERM may not achieve this optimal rate \citep{birge1993rates} for large (so-called non-Donsker) function classes.



We also introduce the following additional notations and definitions: First, $\PP_n$ denotes the (random) uniform measure over $\vec X =(X_1,\ldots,X_n)$. Next, following \cite{bartlett2005local}, we define the lower and upper isometry remainders of $(\cF,\PP)$ for a given $n$. These remainders measure the discrepancy between $L_2(\PP)$ and a ``typical'' $L_2(\PP_n)$,  here ``typical'' means for most of the realizations of $\vec X$.\footnote{These remainders first emerged in the field of metric embeddings, specifically in the definition of quasi-isometries (cf.  \cite{ostrovskii2013metric}).}

In order to introduce these isometry remainders, we first define for each realization of the input $\vec X$, the constants $\cI_{L}(\vec X)$ and $\cI_{U}(\vec X)$ as the minimal numbers $A_{X},B_{X} \geq 0$, respectively, such that the following holds:
\[
    \forall f,g \in \cF:  \    4^{-1}\int(f - g)^2d\PP - A_{\vec X} \leq \int(f - g)^2d\PP_n \leq 4\int(f - g)^2d\PP + B_{\vec X}.
\]
Note that as $A_{\vec X}$ and $B_{\vec X}$ increase, the geometry of $L_2(\PP_n)$ and $L_2(\PP)$ over $\cF$ becomes less similar. For example, in the extreme case of  $A_{\vec X} = B_{\vec X} = 0$, it implies the $L_{2}(\PP)$ and $L_{2}(\PP_n)$ induce the same topology over $\cF$. In words, the lower isometry is the minimal threshold that satisfies the following: all $f,g \in \cF$ that are $\omega(\cI_{L}(\vec X))$ far from each other in $L_2(\PP)$, must be \emph{at least} $\Omega(\|f-g\|)$ far in $L_2(\PP_n)$. The upper isometry remainder implies the converse. To provide further intuition on these remainders, for instance, observe that $\cI_L(\vec X)$ upper bounds on the diameter in $L_2(\PP)$ of possible solutions of ERM; namely, one has
\begin{equation}
    \sup_{\ft \in \cF,\vec \xi \in \R^n}\mathrm{Diam}_{\PP}(\{f \in \cF: \   f|_{\vec X} = \erm|_{\vec X}\})^2 \leq 4 \cdot \cI_{L}(\vec X),
\end{equation}
where $f|_{\vec X} \in \R^n$ is the restriction of $f \in \cF$ on $\vec X$. Finally, the isometry remainders $\cI_{L}(n)$, $\cI_{U}(n)$ are defined as  the ``typical'' values of $\cI_{L}(\vec X)$, $\cI_{U}(\vec X)$:
\begin{definition}\label{D:BrackL}
The lower and upper isometry remainders $\cI_{L}(n)$ and $\cI_{U}(n)$ are defined as the minimal constants $A_n,B_n \geq 0$ (respectively) such that
 \begin{align*}
    \Pr_{\vec X}(\cI_{L}(\vec X) \leq A_{n}, \cI_{U}(\vec X) \leq B_{n}) \ge 1 - n^{-1}.
 \end{align*}
 \end{definition}
\noindent{In the classical regime \citep{van2000empirical}, it is considered to be a standard assumption (such as equivalent entropy and entropy with bracketing numbers) that 
\[
\max \{\cI_{L}(n),\cI_{U}(n) \} \lesssim \nms.
\] 
However, in the high dimensional setting, it may happen that the lower isometry remainder is significantly smaller than the upper isometry remainder, e.g., $\cI_{L}(n) \lesssim \nms $ and $\cI_{U}(n) \gg \nms$ (cf. \cite{liang2020multiple,mendelson2014learning}).

Finally, we remind the reader that $\erm$ is uniquely defined on the data points $\vec X$ when $\cF$ is a convex closed function class, but it \emph{may not} be unique over the entire $\cX$ (as multiple functions in $\cF$ may take the same values at $X_1, \ldots, X_n$). In \S \ref{ss:emp}, the results hold for any possible solution of $\erm$ over $\cX$, whereas in \S \ref{ss:Iso}, we (implicitly) assume that $\erm$ is equipped with a selection rule such that it is also unique over the entire $\cX$ (e.g., choosing the minimal norm solution \citep{hastie2022surprises,bartlett2020benign}); i.e, $\erm: \cD \to \cF$.
\begin{remark}\label{Rem:DiscussLower}
 In the seminal works of Mendelson (cf. \cite{mendelson2017extending} and references within), the small ball condition was introduced to estimate the statistical performance of ERM under less restrictive assumptions as uniform boundedness, Koltchinskii–Pollard entropy condition (cf. \cite{rakhlin2017empirical}) or finite VC-dimension (cf. \cite{mendelson2014learning}). Roughly speaking, under this condition, the lower isometry remainder is relatively small, i.e.  
\[
 \cI_{L}(n) \ll \nms.
\]
 However, it is insufficient to obtain a nice control over the upper isometry remainder, i.e. it may even happen that
\[
    \cI_{U}(n) \asymp 1.
\]
The ideas that appear in the small-ball method suggest that indeed a small lower isometry remainder is a mild assumption over a model $(n,\cF,\PP)$. 

\end{remark}

\subsubsection{Bounding the variance via empirical processes approach}\label{ss:emp}
Here, we assume that the function class and the noise are uniformly bounded.
\begin{assumption}\label{A:boudn}
There exist universal constants $\Gamma_1,\Gamma_2 > 0$ such that $\cF$ is uniformly upper-bounded by $\Gamma_1$, i.e. $\sup_{f \in \cF}\|f\|_{\infty} \leq \Gamma_1$; and the components of $\vec\xi = (\xi_1,\ldots,\xi_n)$ are i.i.d. zero mean with variance one and are almost surely bounded by $\Gamma_2$. 
\end{assumption}
The uniform boundedness assumption on the noise is taken to simplify the proof, which uses Talagrand's inequality. This can be relaxed to i.i.d. sub-Gaussian noise, at the price of a multiplicative factor of $O(\log n)$ in the error term in Theorem \ref{T:Rnd} below. 
 


 \begin{definition}\label{D:eps_U} Set $\eps_{U} := \max\{\nm,\tilde{\eps}\}$, where $\tilde{\eps}$ is the solution of 
\begin{equation}\label{Eq:StatUI}
     \cI_{U}(n) \cdot \log \cN(\eps,\cF,\PP) \asymp n\eps^4.
\end{equation}
 \end{definition}

\noindent{Note that when $\cI_{U}(n) \lesssim \nms$,  $\eps_{U} \asymp \nm$, while if $\cI_U(n, \PP) \gg \nms$ then $\eps_U \gg \nm$.
The following is the  main result of this part:}
\begin{theorem}\label{T:Rnd}
Set $\eps_{V}^2:=\max\{\eps_{U}^2,\cI_{L}(n)\}$, then under Assumptions \ref{A:Convex},\ref{A:Diameter},\ref{A:boudn} the following holds with probability of at least $1 - n^{-1}$:
 \[
    \sup_{f \in \cO_{\delta_n}}\int (f - \E_{\cD}\erm)^2d\PP \lesssim \eps_{V}^2,
 \]
where $\delta_n = O(\eps_{V}^2)$; and in particular   $\cV(\erm,\cF,\PP) \lesssim \eps_V^2.$
\end{theorem}
\noindent{Theorem~\ref{T:Rnd} is a generalization of Theorem~\ref{T:Fix} to the random design case, and its proof uses the strong convexity of the loss and Talagrand's inequality. In \S \ref{ss:ronT} below, we discuss this bound in the context of ``distribution unaware'' estimators. We remark that this Theorem  extends the scope of \cite{caponnetto2006stability} to non-Donsker classes.} 

An immediate and useful corollary of this result is that if we have sufficient control of the upper and lower isometry remainders, the variance will be minimax optimal:
\begin{corollary}\label{C:LUS}
Under Assumptions \ref{A:Convex},\ref{A:Diameter},\ref{A:boudn} and $\max\{\cI_{L}(n), \cI_{U}(n)\} \lesssim \nms$, the following holds: 
 \[
\cV(\erm,\cF,\PP) \lesssim \nms.
 \]
\end{corollary}
In the classical regime, the assumption of $\max\{\cI_{L}(n), \cI_{U}(n)\} \lesssim \nms$ is considered to be standard in the empirical process and shape constraints literature, as it holds for many classical models (see Remark \ref{Rem:Brack} below). 
\subsubsection*{Concluding Remarks}
\begin{remark}\label{rem:lower_iso}
Note that Corollary \ref{C:LUS} may also be derived directly from Theorem \ref{T:Fix} if the noise is assumed to be standard Gaussian. Yet, this corollary holds for any isotropic sub-Gaussian noise -- which is significantly more general.
\end{remark}
\begin{remark}\label{Rem:Brack}
   The assumption of $\max\{\cI_{L}(n),\cI_{U}(n)\} \lesssim \nms$ holds for uniformly bounded classes whose $\eps$-covering numbers are asymptotically equal to the $\eps$-covering numbers with bracketing (see e.g. \cite{van2000empirical,birge1993rates}), which is considered a mild assumption for analyzing ERM on non-parametric and shape-constrained classes. It also holds for classes that satisfy the Koltchinskii-Pollard condition \citep{rakhlin2017empirical} or the $L_2-L_{2+\delta}$ entropy equivalence condition  (see \cite{lecue2013learning} and references therein). In the classical regime, i.e. when $\cF$ is fixed and $n$ grows, it is hard to construct function classes that does not satisfy this assumption for $n$ that is large enough \citep{birge1993rates}. 
\end{remark}
\begin{remark}\label{Rem:Bias}
  Note that a bound similar to that of Theorem \ref{T:Rnd} cannot hold for the bias error term. Indeed, one can construct a class $\cF$ with $\cI_{L}(n) \lesssim \nms$ and $\cV(\erm,\cF,\PP) \asymp \nms$ for which the bias error term $\sup_{\ft \in \cF}B^2(\erm) \asymp 1$, moreover, for this class one has
    \[
        \E_{\vec X,\vec \xi}\left\|\erm - \E_{\vec \xi}\left[\erm |\vec X\right]\right\|_{n}^2 \asymp 1.
    \]
    That is, neither the bias nor the \emph{empirical} variance converge to zero. A remarkable consequence of our results is that even though the ERM only observes the random empirical measure $\PP_n$, its variance, measured in terms of $\PP$, converges to zero when $\cI_{L}(n) \to 0$.
\end{remark}
\subsubsection{Bounding the variance via isoperimetry approach}\label{ss:Iso}
To motivate this part, we point out that just requiring that $\cI_{L}(n) \lesssim \nms$ is considered to be a mild assumption (see Remark \ref{Rem:DiscussLower} above). However, the upper bound in Theorem \ref{T:Rnd} depends on the upper isometry remainder; we would like to find conditions under which this dependency can be removed. Moreover, note that the isometry remainders are connected to the geometry of $(\cF,\PP)$ and not directly to the stability properties of the \emph{estimator}. Using a different approach, based on isoperimetry, we will upper-bound the variance of ERM based on some ``interpretable'' stability parameters of the estimator itself. These stability parameters will be data-dependent relatives of the lower isometry remainder. Unlike the previous part, we do not assume that the function class $\mathcal{F}$ is uniformly bounded by a constant independent of the sample size $n$.

First, we introduce the definition of Lipschitz Concentration Property (LCP):
\begin{definition}\label{Def:LCP}
    Let $\vec Z=(Z_1,\ldots,Z_m)$ be a random vector taking values in $\cZ^{\otimes m}$. $\vec Z$ satisfies the LCP with constant $c_{L} > 0$, with respect to a metric $d:(\cZ^{\otimes m},\cZ^{\otimes m}) \to \R^{+}$, if for all $F: \cZ^m \to \R$ is $1$-Lipschitz, the following holds:
\begin{equation}\label{Eq:CCP}
    \Pr(|F(\vec Z) -  \E F(\vec Z)| \geq t) \leq 2\exp(-c_L t^2).
\end{equation}
\end{definition}
The LCP property is also known as the isoperimetry condition (cf. \cite[\S 1.3]{bubeck2023universal}). Now, we state our first assumption:
 \begin{assumption}\label{A:CCP}
$\vec\xi$ is an isotropic random vector satisfying \eqref{Eq:CCP} with constant $c_{L} = \Theta(1)$, with respect to the Euclidean norm in $\R^n$.
\end{assumption} 
This assumption is stronger than being sub-Gaussian \citep{boucheron2013concentration}, and yet it is significantly less restrictive than requiring normal noise (in which case $c_{L} = 1/2$ \citep{ledoux2001concentration}).
\begin{remark}
Herbst’s argument \cite[\S 3.1.2]{wainwright2019high} implies that the LCP
holds for any random vector $\vec \xi$ satisfying a log-Sobolev inequality; the converse is not true in general. However,
in the seminal work of \cite{milman2009role}, it was shown that if $\vec \xi$ is assumed to be log-concave, then $\vec \xi$ which
satisfies a LCP with constant $C_{L}$ also satisfies a log-Sobolev inequality with constant $\Theta(C_L)$.
\end{remark}

 Recall that $\nm$ is defined as the stationary point of $n\eps^2 \asymp \log \cN(\eps,\cF,\PP)$, and that the conditional variance of $\erm$, which is a function of the realization $\vec X$ of the input, is defined as
\[
 V(\erm|\vec X):= \E_{\vec \xi}[\|\erm - \E_{\vec \xi}[\erm|\vec X]\|^2];
\]
that is, we fix the data points $\vec X$, and take the expectation over the noise.

The formulation of the following definition involves a yet-to-be-defined (large) absolute constant $M > 0$, which will be specified in the proof of Theorem \ref{T:RndLowerIso} (see \S \ref{SP:RndLowerIso} below). We use the notation of $\vec \ft = (f^*(X_1),\ldots,f^*(X_n))$, and $\vec Y = \vec \ft +\vec \xi$.

\begin{definition}\label{D:Brack}
For each realization $\vec X$ and $\ft \in \cF$, let $\rho_{S}(\vec X,\ft)$ be defined as the minimal constant $\delta(n)$ such that
\begin{equation}\label{eq:a_brack}
    \Pr_{\vec \xi}\left\{\vec \xi \in \mathbb R^n: \forall \vec \xi' \in B_n(\vec\xi, M\nm): \|\erm(\vec X, \vec \ft + \vec \xi') - \erm(\vec X, \vec \ft + \vec\xi)\|^2 \le \delta(n)\right\} \geq \exp(-c_2n\nms).
\end{equation}
where $B_n(\vec \xi, r) = \{\vec \xi' \in \mathbb R^n: \|\vec\xi - \vec\xi'\|_n \leq r\}$, and  $c_2 > 0$ is an absolute constant. 
\end{definition}
We set $\rho_{S}(\vec X):= \sup_{\ft \in \cF}\rho_{S}(\vec X,\ft)$. Note that $\rho_{S}(\vec X)$ measures the optimal radius of stability (or ``robustness'') of  $\erm$ to perturbations of the noise \emph{when the underlying function and data points $\vec X$ are fixed}. This is a weaker notion than the lower isometry remainder; in fact, one can verify that $\rho_{S}(\vec X) \lesssim \max\{\cI_L(\vec X),\nms\}$ for every realization $\vec X$ (see Lemma \ref{L:Stable} for completeness). Now, we are ready to present our first theorem:

\begin{theorem}\label{T:RndLowerIso}
    Under Assumptions \ref{A:Convex},\ref{A:Diameter},\ref{A:CCP}, the following holds for every realization $\vec X$ of the data:
    \[
       V(\erm|\vec X) \lesssim \max\{\rho_S(\vec X,\ft),\nms\},
    \]
    and in particular
    $
        \sup_{\ft \in \cF}\E_{\vec X} V(\erm|\vec X) \lesssim \max\{\cI_{L}(n),\nms\}.
    $
\end{theorem}
\noindent{Note that if $\cI_{L}(n) \lesssim \nms$ -- a relatively mild assumption -- then we obtain that the expected conditional variance is minimax optimal. However, we believe that it is impossible to bound the total variance via the lower isometry remainder alone.  Intuitively, $\erm$ only observes a given realization $\vec X$, and in general, the geometry of $\cF$ may ``look different'' under different realizations if $\cI_L(n)$ is large, see \S \ref{ss:ronT} below for further details.} 

In our next result, we identify a  model  which we can bound the \textit{total} variance of $\erm$ by the lower isometry remainder. 
To state the next assumption, we fix a metric $d: \cX \times \cX \to \mathbb R^{+}$ on $\cX$, and denote by $d_n$ the metric on $\cX^n$ given by $d_n(\vec X, \vec X')^2 = \sum d(X_i, X_i')^2$. 
\begin{assumption}\label{A:IsoData}
$\vec X \sim \PP^{\otimes n}$ satisfies \eqref{Eq:CCP} with respect to the metric $d_n(\cdot,\cdot)$, and with constant $c_{X} > 0$ that \emph{only} depends on $\PP$.
\end{assumption}
Note that it is insufficient to assume that $X \sim \PP$ satisfies an LCP, since this does not imply that $\vec X$ satisfies an LCP with a constant independent of $n$ (w.r.t. to $d_n(\cdot,\cdot)$). However, if $X \sim \PP$ satisfies a concentration inequality which tensorizes ``nicely,'' such as a log-Sobolev or $W_2$-transportation cost inequality (cf. \cite[\S 5.2, \S 6.2]{ledoux2001concentration}), then $\vec X \sim \PP^{\otimes n}$ does satisfy this LCP property. 

Next, we assume that with high probability, $\erm$ is at-least almost interpolating the observations:
\begin{assumption}\label{A:Interpolate}
There exist absolute constants $c_I, C_I > 0$, such that the following holds:
\begin{center}
    $
    \Pr_{\cD}\left(n^{-1}\sum_{i=1}^{n}(\erm(X_i) - Y_i)^2 \leq C_I \nms \right) \geq 1 - \exp(-c_I n\nms).
$
\end{center}
\end{assumption}
This assumption is quite common in the study of ``rich'' high-dimensional models, for example, in the setting of \textit{benign overfitting} literature. In this setting, the the function class $\cF$ may depend on $n$ and  is ``large enough'' to interpolate the measurements, see, e.g., \cite{belkin2019does,bartlett2020benign,liang2020just}).

Finally, we introduce another stability notion. Recall the random set $\cO_{\delta} \subset \cF$ of almost-minimizers of the empirical loss, as defined in \eqref{Eq:almost} above; note that, in the random design setting, $\cO_\delta$ depends on both $\vec X$ and $\vec \xi$. The random variable $\diam_{\PP}(\cO_{\delta})$ can be thought of as measuring the stability of the ERM with respect to imprecision in the minimization algorithm (cf. \citep{caponnetto2006stability}). The formulation of the following definition involves another yet-to-be-defined (large) absolute constant $M' > 0$, which will be specified in the proof of Theorem \ref{T:RandomFull} (see \S \ref{ss:rndFull} below), as well as the constant $c_I$ from Assumption \ref{A:Interpolate}.

\begin{definition}\label{Def:rho_O}
$\rho_{\cO}(n,\PP,\ft)$ is defined as the smallest  $\delta(n) \geq 0$ such that  
\begin{equation}\label{eq:rho_O}
\Pr_{\cD}(\diam_{\PP}(\cO_{M'\nms}) \leq \sqrt{\delta(n)}) \ge 2 \exp(-c_I n \nms),
\end{equation}
where $c_{I} \geq 0$ is \emph{the same} absolute constant defined in Assumption \ref{A:Interpolate}.
\end{definition}

In order to understand the relation between this and the previous stability notions, note that under Assumption \ref{A:Interpolate} and the event $\cE$ of Definition \ref{Def:rho_O}, we have that on an event of nonnegligible probability, $\rho_{S}(\vec X,\ft) \le \rho_\cO(n, \PP,\ft)$; in addition, $\rho_{\cO}(n, \PP, \ft) \lesssim \max\{\mathcal I_L(n),\nms\}$, see Lemma \ref{L:Brack} below. Under these additional two assumptions and the last definition, we state our bound for the total variance of $\erm$: 

\begin{theorem}\label{T:RandomFull}
  Under Assumptions \ref{A:Convex},\ref{A:Diameter},\ref{A:CCP}-\ref{A:Interpolate}, the following holds:
    \[        
    V(\erm)\lesssim c_{X}^{-1} \cdot \sup_{\ft \in \cF}\|\ft\|_{\rm Lip} \cdot \max \{\nms,\rho_{\cO}(n,\PP,\ft)\},
    \]
    and in particular one has
$
    \cV(\erm,\cF,\PP) \lesssim c_{X}^{-1} \cdot \sup_{\ft \in \cF}\|\ft\|_{\rm Lip} \cdot \max \{\nms,\cI_{L}(n)\}$.
\end{theorem}
\noindent{Note that  when $\cF$ is a robust learning architecture (i.e. $\cF \subset \{\cX \to \R: \|f\|_{\rm Lip} = O(1)\}$), our bound is optimal. Interestingly, the assumptions of Theorem \ref{T:RandomFull} coincide with those of the model considered in the recent paper of \cite{bubeck2023universal}. Also note that the last theorem connects the total variance of $\erm$ to a ``probabilistic'' threshold for the $L_2(\PP)$-diameter of the \textit{data-dependent} set of $\Theta(\nms)-$approximating solutions of $\erm$. }
\subsubsection*{Concluding Remarks}
\begin{remark}
     One may suspect that the assumptions of almost interpolation and robustness are incompatible, which would render our theorem vacuous. However, perhaps counter-intuitively, in the high-dimensional setting these assumptions can coexist. For example, interpolation with $O(1)$-Lipschitz functions may be possible when the ``intrinsic'' dimension of $\cX$  is $\Omega(\log(n))$ (depending on the richness of $\cF$), though it is generally impossible when the dimension is $o(\log(n))$ (this follows from the behaviour of the entropy numbers of the class of Lipschitz functions; cf. \cite{dudley1999uniform}).
\end{remark}
\begin{remark}\label{Rem:TalLip}
Using Assumptions \ref{A:Convex},\ref{A:Diameter},\ref{A:CCP},\ref{A:IsoData}, one may prove the same bound as in Theorem \ref{T:Rnd}, i.e. that $\cV(\erm, \cF,\PP) \lesssim \eps_{V}^2$, without requiring the noise or the function class to be uniformly bounded. The idea is to obtain the crucial concentration bounds in the proof of Theorem \ref{T:Rnd} by using the LCP properties of $\vec \xi$ and $X\sim \PP$ along with the robustness of $\cF$, rather than via Talagrand's inequality.
\end{remark}
\begin{remark}\label{R:thm_cmp}
    Comparing Theorem \ref{T:RandomFull} to Theorem \ref{T:Rnd}, one sees that if $\cI_L(n) \lesssim \nms$,  the minimax optimality of the variance is implied either by a bound of $\nms$ for the upper isometry constant $\cI_{U}(n)$ or by the Lipschitz and interpolating Assumptions \ref{A:IsoData}-\ref{A:Interpolate}, one may wonder whether the latter set of assumptions actually themselves imply such a bound on $\cI_{U}(n)$.
    
    In fact, the opposite is true: these assumptions are mutually exclusive as soon as the minimax rate is $o(1)$. Indeed, the assumption that the function class is almost interpolating (Assumption \ref{A:Interpolate}) means that $\vec\erm$ closely tracks the observation vector $\vec Y$ (though Assumption \ref{A:Interpolate} only requires this to hold a non-negligible event, the proof of Theorem \ref{T:RandomFull} shows that up to increasing the absolute constant $C$, almost interpolation actually holds with high probability). The variance of $\vec Y$ is bounded below by that of $\vec \xi$, which is $1$ (measured with respect to $\|\cdot\|_n$), which implies easily that the \textit{empirical} variance of $\erm$ is of order $O(1)$ as well.

    On the other hand, a bound of $\nms$ on the upper isometry constant means that up to a multiplicative factor and an additive error of $\nms$, when $\|\erm - \E \erm\|^2$ is small then so is $\|\erm(\vec X,\vec Y) - \E \erm(\vec X,\vec Y)\|_n^2$. Taking expectations, one obtains that the empirical variance of $\erm$ is asymptotically bounded by the population variance plus $\nms$, which is certainly $o(1)$.  
\end{remark}
\subsubsection{Admissibility}
In the final part of this subsection, we study the admissibility of ERM in the random design.

We say that the ERM is \textit{weakly admissible} if there exists $\ft \in \cF$ such that the error of the ERM on such $\ft$ is minimax optimal up to an absolute constant, or equivalently the minimal error of ERM is at most the minimax rate\footnote{In the paper of \cite{kur2021minimal}, a sharp lower bound on the minimal error of ERM in the fixed design setting is proven.}. Note that if we place $\bar{f}_n$ in Corollary \ref{C:A} some  minimax optimal estimator immediately yields that ERM is weakly admissible:
     \begin{equation}\label{Eq:WeakAdmis}
              \inf_{\ft \in \cF} \E_{\cD}\int (\erm - \ft)^2d\PP^{(n)} \lesssim \cM(\cF,\PPf).
     \end{equation}
This definition in the random design is essential, as we may not assume that our estimator is $1$-Lipechitz in the covariates.  In order to state our result, we require the following additional technical assumption:
\begin{assumption}\label{A:Eval}
  The function class $\cF$ is compact with respect to $L_2(\PP)$, and for every $x \in \cX$, the evaluation functional $f \mapsto f(x)$ is continuous in the $L_2(\PP)$ norm when restricted to $\cF$. 
\end{assumption}
As we assumed that $\cF$ is closed in Assumption \ref{A:Convex}, it suffices that $\cF$ have finite $\eps$-entropy for every $\eps$ to ensure that $\cF$ is compact. We will use this regularity condition in order to apply a fixed-point theorem for continuous functions on a compact convex set in a Banach space. 

\begin{theorem}\label{T:RndA}
Under Assumptions \ref{A:Convex},\ref{A:Eval}, we have that
     \[
             \inf_{\ft \in \cF} \E_{\cD} \int (\erm - \ft)^2d\PP \lesssim \max\{ \cV(\erm,\cF,\PP), \cI_{L}(n)\}.
     \] 
\end{theorem}
\noindent{In particular, Theorem \ref{T:RndA} implies that when $\max\{\cV(\erm,\cF,\PP), \cI_{L}(n)\} \lesssim \cM(n,\cF,\PP)$, then ERM is weakly admissible. On the other hand, we conjecture that when $\cI_{L}(n) \gg \cM(n, \cF, \PP)$, ERM \emph{may not} even be weakly admissible. This would follow from the stronger conjecture that the bound of Theorem \ref{T:RndA} is optimal.}

 \begin{remark}
The assumption that the evaluation functional is continuous in $L_2(\PP)$ may seem restrictive. In fact, though, the proof of Theorem \ref{T:RndA} also goes through if there exists a stronger norm $\|\cdot\|'$ on $\cF$ than the $L^2(\PP)$ norm such that $\cF$ is compact and the evaluation functionals $f \mapsto f(x)$ are continuous with respect to the topology induced by $\|\cdot\|'$. Natural examples of such $\cF, \|\cdot\|'$ are Sobolev space. For simplicity, we have stated the theorem under Assumption \ref{A:Eval}.
\end{remark}

\subsection{On the landscape of ERM in the non-Donsker regime}\label{SS:Jagged}
Finally, we establish a counter-intuitive behavior of the landscape around $\erm$ for various non-parametric models that lie in the non-Donsker regime. For our purposes, the ``non-Donsker regime'' simply means that the model satisfies Assumption \ref{A:ND} below . The conditions in Assumption \ref{A:ND} may seem a bit technical at first glance, but they cover many well-studied non-parametric models that appear in the shape-constraints literature, including convex/bounded $\alpha$-H\"older regression in the suitable dimensions. 
\begin{assumption}\label{A:ND}
  The model $(\cF,\PP)$ satisfies the following:
   \begin{enumerate}
       \item $\cF$ is uniformly bounded by an absolute constant $\Gamma > 0$.
       \item The function  $\eps \mapsto \frac{\eps^2}{\log(\eps^{-1})}\cdot\log \cN(\eps,\cF,\PP)$ is decreasing in $\eps \in (0,\Gamma)$.
       \item The lower and upper isometry remainders satisfy: $\cI_{L}(n) = o(\nms)$ and $\cI_{U}(n) = O(\nms)$ .
   \end{enumerate}
\end{assumption}

Now, we are ready to state our result:
\begin{theorem}\label{C:Jagged}
     Let $(\cF,\PP,\vec \xi)$ that satisfies Assumptions \ref{A:Convex},\ref{A:normal},\ref{A:ND} and set $\nm = \eps(n)$ as in \eqref{Eq:minimax} above. Then, there exists a sequence $C_{\cF,\PP}(n) = \omega(1)$ and a sequence of functions $\ft = \ft(n) \in \cF$, such that $\E_{\cD}\|\erm - \ft\|^2 \lesssim \nms$ (i.e., each $\ft$ is weakly admissible) and
    \begin{equation}\label{Eq:admsta}
    \crl*{f \in \cF: \int (f-\erm)^2 d\PP \lesssim \nms} \not\subset \cO_{C_{\cF,\PP}(n)\cdot \nms},
    \end{equation}
     with probability of at least $1-n^{-1}$.
\end{theorem}
\noindent{Theorem \ref{C:Jagged} says that for some target function,  $\erm$ displays counterintuitive behavior: on the one hand, $\erm$ estimates $\ft$ optimally, but on the other hand, for most $\vec{\xi}$ there exist functions which are very close to $\erm$ in $L_2(\PP)$, and yet far from being minimizers of the squared error.}

\section{Discussion} \label{ss:ronT}
In order to start this discussion, recall Theorem \ref{T:Rnd} above. Note that when $\max\{\cI_{L}(n),\cI_{U}(n)\} \gg \nms$, it implies that we have that $\eps_{V}^2 \gg \nms$, i.e., the upper bound is larger than the minimax rate. Therefore, this bound seems at first glance to be suboptimal.

To the best of our knowledge, all estimators that attain the minimax rate, such as aggregation and related algorithms (cf. \cite{yang2004aggregating}), depend on the marginal distribution $\PP$ of the covariates. In many cases, though, we do not know or have oracle access to the marginal distribution $\PP$, and the estimator only has access to $\PP_n$ and to $\cF$. It is natural to ask what is the minimax rate of ``distribution-unaware'' estimators that only depend on $\PP_n$ and the function class $\cF$, when the underlying distribution $\PP$ is allowed to vary over some family of distributions.

To this end, given some family of probability distributions $\cP$ on a domain $\cX$, consider the following measurement of optimality of an estimator:
$$\Delta_{(du)}(n,\cF,\cP) = \inf_{\bar f_n} \sup_{\QQ \in \cP} \frac{\cR(\bar f_n,\cF,\QQ)}{\cM(n,\cF,\QQ)}.$$
 We say that there exists an \emph{optimal} distribution unaware estimator over $(n,\cF, \cP)$ when $\Delta_{(du)}(n,\cF,\cP) = \Theta(1)$.

Unsurprisingly, suppose we do not place additional assumptions on $\cF$ and $\cP$ (beyond convexity). In that case, it may happen that $\Delta_{(du)}(n,\cF,\cP)=\omega(1)$ -- i.e. no single estimator attains the minimax error on every distribution $\QQ \in \cP$. In other words, a minimax optimal estimator for $\QQ$ must ``know'' $\QQ$. In fact, one may construct a set of probability distributions $\cP$ on a domain $\cX$ and a function class $\cF$ such that for any $\QQ \in \cP$, $\mathcal M(n,\cF, \QQ) = O(n^{-1})$ (the parametric rate), and for any estimator $\bar{f}_n$, one may find $\QQ \in \cP$ such that $\cR(\bar{f}_n,\cF,\QQ) = \Theta(1)$; and in particular $\Delta_{(du)}(n,\cF,\cP) = \Theta(n)$ (see Example \ref{Ex:One} below). 

It's also intuitively clear that the version space diameter, namely, 
\[
    \Psi(n,\PP):= \sup_{\ft \in \cF} \mathbb E_{\vec X}[\diam_{\PP}(\{f \in \cF:  \ft|_{\vec X} = f|_{\vec X}\})]
\]
should appear in the error of any ``distribution-unaware'' estimator in terms of $L_2(\PP)$ (though we do not know how to show this in complete generality). Clearly, for every model, $\Psi(n,\PP) \leq \cI_{L}(n)$. Therefore, it is not surprising that the bound of Theorem \ref{T:Rnd} includes the lower isometry remainder.  The upper isometry remainder $\cI_{U}(n)$, though, is not tightly connected to $\Psi(n,\PP)$. Nonetheless, we conjecture that it cannot be removed from the bounds of Theorem \ref{T:Rnd}. Specifically, we propose the following conjecture:

\begin{conjecture}\label{Conj:UI}
For every $n \geq 1$ there exists models  $\cF$ and a distribution $\PP$ in which their corresponding ERM satisfies
\[
    \cV(\erm,\cF,\PP) \asymp \eps_V^2.
\]
\end{conjecture}
This conjecture implies that the bound of Theorem \ref{T:Rnd} cannot be improved \textbf{without} additional assumptions. The intuition behind this conjecture is as follows: the ERM sees the geometry of 
\[
    \cF_n = \{(f(X_1),\ldots,f(X_n)): f \in \cF\}
\]
and perturbing the data points $X_1,\ldots,X_n$ in ``adversarial'' way by some small $\delta_1,\ldots,\delta_n \in \cX$ may change the geometry of $\cF_n$, and it will reduce the ``stability'' of $\erm$ in terms of $L_2(\PP)$. It mainly follows from the fact that $\erm$ may not be a Lipschitz function in the data $\vec X$ with respect to the $L_2(\PP)$-norm, in contrast to its Lipschitzness in the observations (i.e, with respect to the $L_2(\PPf)$-norm). 

Our confidence that this is the correct explanation for the appearance of $\cI_{U}(n)$, rather than some other phenomenon, derives from Theorem \ref{T:RndLowerIso},  which precisely states that the expected conditional variance of $\erm$ is upper bounded by the lower isometry radius, i.e.
\begin{equation}\label{Eq:tv1}
    \E_{\vec X}V(\erm|\vec X) \lesssim \cI_{L}(n).
\end{equation}

Therefore, if Conjecture \ref{Conj:UI} is correct and there are models in which $\cV(\erm,\cF,\PP) \gtrsim \eps_U^2 \gg \cI_L(n)$, this must be due to the variance of conditional expectations:
\begin{equation}\label{Eq:tv2}
V\left(\E_{\vec \xi}\left[\erm|\vec X\right]\right)\gtrsim \eps_{U}^2,
\end{equation}
and $V\left(\E_{\vec \xi}\left[\erm|\vec X\right]\right)$ is precisely the error term which captures how the geometry of $\cF$ varies under different realizations. 


\section{Proof sketches}\label{S:SP}
In this section, we sketch proofs of less technical results to give the reader a flavor of our methods. The full proofs are given in the next section. For the proofs we introduce some additional notations. For $m \in \mathbb{N}$, we set $[m]:= \{1,\ldots,m\}$. The inner products in $L_2(\PP), L_2(\PPf)$ are denoted by $\langle \cdot, \cdot \rangle, \langle \cdot, \cdot\rangle_n$, respectively.  In the fixed design proofs, with some ambiguity of notation, $\erm,\ft$ are observations vectors in $\R^n$.

\subsection{Sketch of proof of Theorem \ref{T:Fix}}\label{SP:Fix}
Here, we sketch a simple proof of a weaker version of our result, namely $V(\erm) \lesssim \cM(\cF, \PPf)$, under the stronger assumption that 
\begin{equation}\label{Eq:SimplifiedModel}
    \cM(\cF,\PPf) \asymp \nms,
\end{equation}
where $\nm$ solves $\log \cN(\eps,\cF,\PPf) \asymp n\eps^2$ (this holds under reasonable assumptions on $\cF$, but can be dispensed with; see Lemma \ref{Lem:Minimax} for the exact characterization). In \S \ref{sss:TfixSimple}, we fill in the details of this sketch, and in \S \ref{sss:TfixFull}, we give the full proof of Theorem \ref{T:Fix}. 

The proof uses the probabilistic method \citep{alon2016probabilistic}. Let $f_1,\ldots,f_N$ be centers of a minimal $\epsilon_*$-cover of $\cF$ with respect to $L_2(\PP^{(n)})$, per Definition~\ref{Def:Covering}. First, since for any $i\in[N]$, the map $\vec{\xi}\mapsto \sqrt{n} \cdot \|\erm(\vec{\xi}) - f_i\|_n$ is $1$-Lipschitz, \eqref{Eq:CCP} and a union bound ensure that with probability at least $1-\frac{1}{2N}$, for all $i\in[N]$, 
    $$
     \E_{\vec\xi}\|\erm - f_{i}\|_{n} - \|\erm - f_{i}\|_{n} \lesssim \sqrt{\frac{\log N}{n}}.
    $$
    On the other hand, by the pigeonhole principle, there exists at least one $i^*\in[N]$ such that with probability at least $1/N$, 
    $\|\erm - f_{i^*}\|_{n} \leq \eps_*$. Hence, there exists at least one realization of $\vec{\xi} \in \R^{n}$ for which both bounds hold, and thus,  \textit{deterministically},
    $$\E_{\vec\xi}\|\erm - f_{i^*}\|_{n} \lesssim \eps_* + \sqrt{\frac{\log N}{n}} \lesssim \eps_*$$
    where we used the balancing equation \eqref{Eq:minimax}. Another application of \eqref{Eq:CCP} and integration of tails yields
    $$
    V(\erm) = \E_{\cD}\|\erm - \E_{\cD} \erm\|_{n}^2 \leq \E_{\cD}\|\erm - f_{i^*}\|_{n}^2 \lesssim \eps_*^2,
    $$
    implying that the variance of ERM is minimax optimal. 
    
\subsection{Sketch of proof of Theorem \ref{T:RndLowerIso}}\label{SP:RndLowerIso}
As is well-known, the Lipschitz concentration condition \eqref{Eq:CCP} is equivalent to an isoperimetric phenomenon: for any set $A \subset \mathbb R^n$ with $\Pr_{\vec \xi}(A) \geq 1/2$, its $t$-neighborhood $A_t = \{\vec \xi \in \R^n: \inf_{x\in A}\|x-\vec \xi\|_{n} \leq t \}$ satisfies 
\begin{equation}
\Pr_{\vec \xi}(A_t) \geq 1- 2\exp(-nt^2/2). 
\end{equation}
One sees quickly that this implies that if $A$ has measure at least $2 \exp(-nt^2/2)$, then $A_{2t}$ has measure $1 - 2 \exp(-nt^2/2)$.

Let $\mathcal E$ be the event of Definition \ref{D:Brack}. As in \S \ref{SP:Fix} above, one obtains via the pigeonhole principle and the definition of $\nm$ that there exists some $f_c \in \cF$ such that
\[
    \Pr_{\vec \xi}(\underbrace{\{\erm \in B(f_c,\nm)\} \cap \cE }_{A}|\vec X) \geq \frac{\Pr_{\vec \xi}(\cE)}{\cN(\nm,\cF,\PP)} \geq \exp(-C_3 n\nms).
\]
By isoperimetry, $\Pr_{\vec \xi}(A_{2t}) \ge 1 - 2\exp(-n t^2/2)$, where $t = M\nm/2$ and $M$ is chosen such that $(M/2)^2 \ge 2 C_3$; this fixes the value of the absolute constant $M$ used in \eqref{eq:a_brack}.

Applying \eqref{eq:a_brack} yields that if $\vec \xi \in A \subset \cE$ and $\|\vec \xi' - \vec \xi\|_n \le M \nm = 2t$, $\|\erm(\vec \xi) - \erm(\vec \xi')\| \le \rho_S(\vec X, \ft)$ and so $\|\erm(\vec \xi') - f_c\| \le \nm + \rho_{S}(\vec X,\ft)$. This implies
\[
    \Pr_{\vec \xi}(\{\erm \in B(f_c, \nm + \rho_{S}(\vec X,\ft))\}|\vec X) \geq \Pr_{\vec \xi}(A_{2t}|\vec X) \geq 1-2\exp(-n t^2/2),
\]
which implies via conditional expectation that
$
    V(\erm|\vec X) \lesssim \max\{\rho_{S}(\vec X),\nms\}
$, as desired (where we used that $\nms \gtrsim \log(n)/n$ (see Lemma \ref{Lem:DumTech} below), and therefore $\exp(-nt^2) = O(\nms)$).
\subsection{Sketch of Proof of Theorem \ref{C:A}}\label{ss:pChat}
Here, we prove this result when $\cF$ is conve and compact with respect to $L_2(\PPf)$. We denote by
\[
    \cF_n:= \{(f(x_1),\ldots,f(x_n)): f \in \cF\} \subset \R^n.
\]
For completeness, we provide a proof without this assumption in \S \ref{ss:rdchat} below.
Consider the map $F: \cF_n \to \cF_n$ defined via
\[
    \ft \to \E_{\vec \xi}\erm,
\]
i.e., $\ft$ maps to the expectation of the $\erm$ when the underlying function is $\ft \in \cF$. One easily verifies that $F$ is continuous, since projection to a convex set is a $1$-Lipschitz function. In addition, the convexity of $\cF_n$ implies that $F(\ft) \in \cF_n$ for all $\ft \in \cF_n$. Thus $F$ is a continuous map from the compact convex set  $\cF_n$ to itself, so by the Brouwer fixed point theorem, there exists an $\ft \in \cF_n$ such that
\[
   \ft = \E_{\vec \xi} \erm,
\]
i.e. on this $\ft$, it holds that  $B^2(\erm) = 0$. Let $\Ef^2 = \E_{\vec \xi} \|\erm - \ft\|_n^2 = V(\erm)$, now recall that $\erm: \R^n \to \R^n$ is the projection to $\cF_n$, which is a contraction in its  input (w.r.t. to $\|\cdot\|_n$). Then, we know that for any $\vec \xi$ and 
$f \in \cH_{*}$, it holds that 
\begin{align*}\|\erm(f + \vec \xi) - f\|_n^2 &\le 3(\|\erm(f + \vec \xi) - \erm(\ft +  \vec \xi)\|_n^2 + \|\erm(\ft + \vec \xi) - \ft\|_n^2 + \|\ft - f\|_n^2) \\
 &\le 3(\|\erm(\ft + \vec \xi) - \ft\|_n^2 + 8 \Ef^2),
\end{align*}
and taking the expectation over $\vec \xi$ we see that $\E_{\vec \xi} \|\erm(f + \vec \xi) - f\|_n^2$, the squared error for any ERM when the underlying function is $f \in \cH_{*}$, is at most $27 \cdot \Ef^2$. Now, let $\widehat{h}_n: \mathbb R^n \to \R^n$ be any estimator. By Theorem \ref{T:Fix}, we have 
$$\Ef^2 \lesssim  \max_{f \in \cH_{*}} \|\widehat{h}_n(f + \vec \xi) - f\|_n^2.$$
Picking $f \in \cH_{*}$ which maximizes the error of $\widehat{h}_n$, we have that the squared error of $\erm$ on $f$ is upper-bounded by $c \cdot \Ef^2$ and the squared error of $\widehat{h}_n$ on $f$ is lower-bounded by $c_1 \cdot  \Ef^2$, which is precisely what we want.

\paragraph{Acknowledgements:}
This work was supported by the Simons Foundation through Award 814639 for the Collaboration on the Theoretical Foundations of Deep Learning, the ERC under the European Union’s Horizon 2020 research and innovation programme (grant agreement No 770127), and the NSF (awards DMS-2031883, DMS-1953181). Part of this work was carried out while the first two authors were in residence at the Institute for Computational and Experimental Research in Mathematics in Providence, RI, during the Harmonic Analysis and Convexity program; this residency was supported by the NSF (grant DMS-1929284). Finally, the first two authors also wish to acknowledge Shiri Artstein-Avidan for introducing them to each other. We also acknowledge Reese Pathak and Nikita Zhivotovskiy for their helpful suggestions for this paper. 
\bibliography{Bib}
\section{Proofs}\label{S:pfs}
 We begin with additional notation: Given $x_1,\ldots, x_n$ and $h : \cX \to \R$, we denote
\[
\Xf{h} = n^{-1}\sum_{i=1}^{n}\xi_ih(x_i).
\] 
For $g \in \cF$, we set $B_n(g,t):= \{f\in\cF:\|f-g\|_{n} \leq t\}$, and $B(g,t):= \{ f\in\cF:\|f-g\| \leq t\}$. 
Throughout the proof, we denote by $c,c_1,c_2,\ldots \in (0,1)$, an $C,C_1,\ldots \geq 0$ absolute constants (not depending on $\cF$ or on $n$) that may change from line to line.  

\subsection{Proof of Theorem \ref{T:Fix}}\label{SS:Fix}

In \S \ref{sss:TfixSimple}, we fill in the details the proof sketch that was given above under the additional assumption of \eqref{Eq:SimplifiedModel}, and in \S \ref{sss:TfixFull}, we give the full proof without additional assumptions. We remark that our proof holds for any noise that satisfies the LCP property \eqref{Eq:CCP} defined above. 

 \subsubsection{Proof of Theorem \ref{T:Fix} under \eqref{Eq:SimplifiedModel}}\label{sss:TfixSimple} First, we show that for all $t \geq 0$ and for any fixed $f \in \cF$, the following holds:
\begin{equation}\label{Eq:concetProj}
    \Pr_{\vec \xi}\crl*{ \left|\|\erm - f\|_n - \E_{\vec \xi}\|\erm - f\|_n \right| \geq t } \leq 2\exp(-c_Lnt^2),
\end{equation}
Indeed, this will follow immediately from the LCP condition \eqref{Eq:CCP} with $c_{L} = n^{-0.5}$ if we prove that  $h(\vec \xi) = \|\erm(\vec \xi) - f^{*}\|_{n}$ is a $n^{-0.5}$-Lipschitz function. 

To prove this claim, observe that $\erm(\vec \xi)$ is the projection of $\vec Y = \ft + \vec \xi$ onto the convex set
 \[
    \cF_n:=\{(f(x_1),\ldots,f(x_n)): f \in \cF\} \subset \R^n.
\]

Therefore, we obtain
\begin{align*}
  |h(\vec\xi_1)-h(\vec\xi_2)| = |\|\erm(\vec\xi_1) - f^{*}\|_{n} - \|\erm(\vec\xi_2) - f^{*}\|_{n}| &\leq \|\erm(\vec\xi_1) - \erm(\vec\xi_2)\|_{n} \\&\leq \|\vec\xi_1-\vec\xi_2\|_{n} = n^{-1/2}\|\vec\xi_1-\vec\xi_2\|_{2},
\end{align*}
where we have used the fact that the projection to a convex set is a contracting operator. This concludes the proof of \eqref{Eq:concetProj}.

Next, fix $\epsilon > 0$ (to be chosen later), let $\cN(\eps):= \cN(\eps,\cF,\PPf)$, and let $A = \{f_1,\ldots,f_{\cN(\eps)}\}$ be a minimal $\eps$-net of $\cF$. By the pigeonhole principle, there exists at least one element $f_{\eps} \in A$ such that
\begin{equation}\label{Eq:Pigeon}
    \Pr( \|\erm - f_{\eps}\|_{n} \leq \eps) \geq 1/\cN(\eps).
\end{equation}

Also, setting $f = f_\eps$ in \eqref{Eq:concetProj} we have
\begin{equation}\label{Eq:concetNet}
    \Pr\left( |\|\erm - f_{\eps}\|_{n}  - \E_{\vec \xi}\|\erm - f_{\eps}\|_{n})| \geq t\right) \leq 2\exp(-c_{L}nt^2).
\end{equation}
Taking $t = \log(4)\sqrt{\frac{\log \cN(\eps)}{c_{L}n}}$ in \eqref{Eq:concetNet} yields 
\begin{equation}\label{Eq:LastFix}
    \Pr\lp \big|\|\erm - f_{\eps}\|_{n}  - \E_{\vec \xi}\|\erm - f_{\eps}\|_{n}\big| \geq \log(4)\sqrt{\frac{\log \cN(\eps)}{c_Ln}} \rp \leq \frac{1}{2\cN(\eps)}.
\end{equation}
Combining \eqref{Eq:Pigeon} and \eqref{Eq:LastFix} via the union bound we obtain
\[
  \Pr\left( \|\erm - f_{\eps}\|_{n} \leq \eps , \left|\|\erm - f_{\eps}\|_{n}  - \E_{\vec \xi}\|\erm - f_{\eps}\|_{n}\right| \leq \log(4)\sqrt{\frac{\log \cN(\eps)}{c_Ln}}\right) \geq \frac{1}{2|\cN(\eps)|} > 0
\]
Since the event of the last equation holds with positive probability, we must have
\[
     \E_{\vec \xi}\|\erm - f_\eps\|_{n} \leq \eps  +\log(4) \sqrt{\frac{\log \cN(\eps)}{c_{L}n}}.
\]
To optimize the RHS over $\epsilon$, we take $\eps$ such that $\eps = \sqrt{\log \cN(\eps)/(c_L n)}$ --- i.e., $\eps = \nm$ --- and get
 \[
    \E_{\vec \xi}\|\erm - f_\eps\|_{n} \leq C\nm/\sqrt{c_{L}}.
 \]
Substituting in \eqref{Eq:concetNet} and taking $t = \nm$, we obtain
\[
    \Pr(|\|\erm - f_\eps \|_{n} \geq C\nm/\sqrt{c_L}) \leq 2\exp(-c n \nm^2).
\]
This easily implies that $\E[\|\erm - f_\eps \|_{n}^2] \le C_1\nms/c_{L}$, and therefore also 
\[
\E[\|\erm - \E\erm \|_{n}] \le (\E[\|\erm - \E\erm \|_{n}^2])^{1/2} \le(\E[\|\erm - f_\eps \|_{n}^2])^{1/2} \le C_2\nm/\sqrt{c_L}.
\]
Applying \eqref{Eq:concetProj} once again, now with $f = \E\erm$, we obtain
\begin{equation}
\Pr(\|\erm - \E \erm\|_{n}^2 \geq C_2\nms/c_{L}) \leq 2\exp(-c n \nms).
\end{equation}

\subsubsection{Full proof of Theorem \ref{T:Fix}}\label{sss:TfixFull}
First note that for any class $\cH$, we have $\mathcal M(\cH, \PPf) \le \diam_{\PPf}(\cH)^2$ (consider a constant estimator),and applying this to $\cH_*$ yields $\mathcal M(\cH_*, \PPf) \le 4 V(\erm)$. Thus we need only prove the nontrivial inequality $\mathcal M(\cH_*, \PPf) \gtrsim V(\erm)$.

We will consider two cases: First, when $\cM(\cH,\PPf) \leq Sn^{-1}$ for sufficiently large $S \geq 0$ (i.e. the parametric case), the result follows from classical theory. The remaining case will be handled in similar fashion as in \S \ref{sss:TfixSimple} above, but with a more careful analysis.

\paragraph{Case I: $ V(\erm) \leq Sn^{-1}$.} Certainly, there exists $g \in \cF$ such that $\|g - \E \erm\|_n^2 \ge V(\erm)$, which implies, by the convexity of $\cF$ that there exists $h \in \cH_*$ with $\sqrt{V(\erm)} \le \|h - \E \erm\|_n \le 2\sqrt{S/n}$. Applying the two-point method to $h$ and $\mathbb E\erm$ (see e.g., \citep[Example 15.4]{wainwright2019high}), one sees easily that the minimax rate of $\cH_*$ is $\Omega(V(\erm))$.

\paragraph{Case II: $V(\erm) \geq Sn^{-1}$.} 
To treat this case, we use the following characterization of the minimax rate in the fixed design setting (cf. \citep[Thm 2.11]{neykov2022minimax}):

\begin{lemma}\label{Lem:Minimax}
Under Assumptions \ref{A:Convex},\ref{A:normal}, $\cM(\cH,\PPf) \asymp \nms$, where $\nm$ solves the equation
   \begin{equation}\label{Eq:MinimaxF}
    \log \cN^{\text{loc}}(\eps,\cH,\PPf) \asymp n\eps^2,
\end{equation}
where $\cN^{\text{loc}}(\eps,\cF,\PPf) = \sup_{f \in \cF}\cN(\eps/4,\{ h \in \cH: \|h - f\|_n \leq \eps\},\PPf)$.
\end{lemma}

By the lemma, it suffices to show that 
$$\log \mathcal N^{loc}(2\sqrt{V(\erm)}, \cH_*, \PPf) \gtrsim n V(\erm),$$
as this will imply that $2\sqrt{V(\erm)} \le \nm$ and hence $V(\erm) \lesssim \mathcal M(\cH_*, \PPf)$. Noting that 
$$\{h \in \cH_*: \|h - \E \erm\|_n \leq 2\sqrt{V(\erm)}\} = \cH_*$$ 
we have 
$\cN^{\text{loc}}(\sqrt{V(\erm)}/2,\cF,\PPf) \ge \cN(\sqrt{V(\erm)}/2, \cH_*, \PPf)$ and hence 
it suffices to show that 
\begin{equation}\label{Eq:LYBA}
\log \cN(\sqrt{V(\erm)}/2, \cH_*, \PPf) \ge c_1 n \cdot V(\erm)
\end{equation}
for an appropriate $c_1 > 0$ to be chosen later.

Suppose to the contrary that
$\log N \le c_1 n \cdot V(\erm)$, where $N = \cN(\sqrt{V(\erm)}/2, \cH_*, \PPf)$.

We consider the distribution of $\erm$ when the true function is $\ft$. First, note that as $\E \|\erm - \E \erm\|^2 = V(\erm)$, we have that
\begin{equation*}
    \Pr(\erm \in \cH_{*}) = \Pr(\erm \in B_n(\E_{\cD}\erm, 2\sqrt{V(\erm)})) = 1 - \Pr(\|\erm - \E_{\cD}\erm\|_{n}^2 \geq 4V(\erm)) \geq 3/4
\end{equation*}
by Chebyshev's inequality. Let $A = \{f_1, \ldots, f_{\cN}\}$ be a minimal $\sqrt{V(\erm)}/2$-net in $\cH_{*}$; by the pigeonhole principle, there exists at least one element $g \in A$ such that 
\begin{equation}\label{Eq:loc_erm_pigA}
\Pr(\|\erm - g\|_{n} \le \sqrt{V(\erm)}/2) \ge \frac{3}{4N} \ge  3\exp(-c_1 nV(\erm))/4.
\end{equation}
Next, we apply \eqref{Eq:concetProj} with $f = g$ and $t = \sqrt{V(\erm)}/6$, to obtain 
\begin{equation}\label{Eq:loc_erm_concA}
\Pr_{\vec \xi} \left(|\|\erm - g\|_n - \E_{\vec \xi}\|\erm - g\|_n| \leq \sqrt{V(\erm)}/6\right) \geq 1-2\exp(-nV(\erm)/18).
\end{equation}
Recalling that we are in the case $V(\erm) \ge S n^{-1}$, by choosing $c_1 > 0$ small enough and $S > 0$ large enough we can ensure that $\exp(n V(\erm) (1/18 - c_1)) > 8/3$, or equivalently
\begin{equation} \label{Eq:pig_balanceA}
\frac{3}{4}\exp(-c_1 nV(\erm)) -  2\exp(-n V(\erm)/18) > 0.
\end{equation}
Combining \eqref{Eq:loc_erm_pigA}, \eqref{Eq:loc_erm_concA}, and \eqref{Eq:pig_balanceA} yields
$$\Pr(\|\erm - g\|_n \le \sqrt{V(\erm)}/2) + \Pr \left(\left|\|\erm - g\|_n - \E_{\vec \xi}\|\erm - g\|_n \right| < \sqrt{V(\erm)}/6\right) > 1,$$
so the two events \[
\left\{\left|\|\erm - g\|_n - \E_{\vec \xi}\|\erm - g\|_n\right| < \sqrt{V(\erm)}/6\right\}, \{\|\erm - g\|_n \le \sqrt{V(\erm)}/2\}
\] have nonempty intersection, which implies that $\E_{\cD}\|\erm - g\|_n < 2\sqrt{V(\erm)}/3$. 

Let $h(\vec \xi) = \|\erm - g\|_n$. We have 
$\E h^2 = (\E h)^2 + \E(h - \E h)^2 < 4V(\erm)/9$. As $h$ is $\frac{1}{\sqrt n}$-Lipschitz, the LCP implies that $h$ is $\frac{1}{\sqrt n}$-subgaussian. Thus $h - \E h$ is a centered $\frac{1}{\sqrt n}$-subgaussian random variable, so $\E(h - \E h)^2 \le \frac{2}{n}$ \citep[Proposition 2.5.2]{vershynin2018high}, and hence
\begin{equation*}
    \E_{\cD} \|\erm - g\|^2_n <  \frac{4}{9}V(\erm) + \frac{2}{n}.
\end{equation*}
Again recalling that $V(\erm) > S n^{-1}$, by taking $S$ large enough we can ensure that $\E_{\vec\xi} \|\erm - \E \erm \|^2_n < V(\erm)$, which contradicts the definition of $V(\erm)$.

It remains to prove the last statement of the theorem, namely that $\sup_{f \in \cO_{\delta_n}} \|f - \E\erm\|_n^2 \asymp \cM(\cH_*, \PPf))$ with high probability. We have seen that $\E\|\erm - \E\erm\|_n^2 \lesssim \cM(\cH_*, \PPf)$, so $\E\|\erm - \E \erm\|_n \le C \sqrt{\cM(\cH_*, \PPf)}$. Applying \eqref{Eq:concetProj} once again with $t = \sqrt{\cM(\cH_*,\PPf)}$, we have 
\begin{equation}\label{Eq:erm_var}
\P(\|\erm - \E \erm\|_n^2 \ge C \cM(\cH_*, \PPf)) \le 2\exp(-cn \cM(\cH_*, \PPf)).
\end{equation}

Condition on the high-probability event of \eqref{Eq:erm_var} above, and consider some $f \in  \cO_{\delta_n}$. Since
\[\|f - \E\erm\|_n^2 \le 2(\|f - \erm\|_n^2 + \|\erm - \E \erm\|_n^2),\]
to obtain the theorem it suffices to show that for any $f \in \cO_{\delta_n}$, we have
\[
\|f - \erm\|_{n}^2 \le \delta_n^2
\]
deterministically.

This is a matter of elementary convex geometry: we know that $\erm$ is the closest point in the convex set $\mathcal F_n$ to the point $\vec Y$, which implies that the ball $B = B(\vec Y, \|\erm - \vec Y\|_n)$ is tangent to $\mathcal F_n$ at $\erm$. This implies that $\mathcal F_n$ is contained within the positive half-space $H^+$ defined by the supporting hyperplane of $B$ at $\erm$, i.e.,
\begin{equation}\label{Eq:erm_hyp}
\mathcal F_n \subset H^+ = \{f: \langle \erm - \vec Y, f - \vec Y\rangle \ge \|\erm - \vec Y\|^2\}.
\end{equation}
We now compute:
\begin{align*}
\|f - \vec Y\|_n^2 = \|f - \erm\|_n^2 + \|\erm - \vec Y\|_n^2 + 2 \langle \erm - \vec Y, f - \erm\rangle_n.
\end{align*}
Since $f \in \mathcal F_n$, \eqref{Eq:erm_hyp} implies that $\langle f - \vec Y, \erm - \vec Y\rangle_n \ge \langle \erm - \vec Y, \erm - \vec Y\rangle_n$, or equivalently, $\langle f - \erm, \erm - \vec Y\rangle_n \ge 0$. Hence we obtain
\[
\|f - \erm\|_n^2 \le \|f - \vec Y\|_n^2 - \|\erm - \vec Y\|_n^2,
\]
but the RHS is at most $\delta_n$ by the definition of $\cO_{\delta_n}$. This concludes the proof. 

\subsection{Proof of Proposition \ref{T:FixLower}}
By the definition of the minimax risk, there exists some $\ft \in \cF$ with risk at least $\delta^2:= \cM(\cF,\PPf)$. By translating $\mathcal F$, we may assume $\ft = 0$ without loss of generality, so that $\E_{\vec \xi}[\|\erm\|_n^2] \gtrsim \delta^2$. 

Write $\erm(\vec \xi)$ for the ERM computed when the target function is $\ft = 0$ and the noise is $\vec \xi$, namely, the projection of $\vec \xi$ onto $\mathcal F$. We wish to show that $\E_{\vec \xi}[\|\widehat{f}_n - \E \erm\|_n^2] \gtrsim \delta^4$.

The fact that $\erm(\vec \xi)$ is the projection of the observation vector $\vec \xi$ on the  convex set $\cF$ implies, by convexity, that $\langle f - \erm, \erm - \vec \xi\rangle_n \ge 0$ for any $f \in \cF$ (see \S\ref{SS:Fix} for the easy argument). Substituting $f = \ft = 0$ and rearranging immediately yields that for any $\vec \xi$,
$$\langle \erm(\vec \xi), \vec \xi\rangle_n \ge \|\erm(\vec \xi)\|_n^2.$$
Write $f_e = \mathbb E_{\vec \xi} \widehat{f}_n(\vec \xi)$. Since $\mathbb E_{\vec \xi} \vec \xi = 0$, we may take expectations and insert $\widehat{f}_e$ to obtain
$$\E_{\vec \xi} \langle \erm(\vec \xi) - f_e, \vec \xi\rangle_n \ge \E_{\vec \xi} \|\widehat{f}_n(\vec \xi)\|_n^2 \ge \delta^2.$$
Applying Cauchy-Schwarz, we obtain 
$$\E_{\vec \xi}\|\erm(\vec \xi) - f_e\|^2_n \cdot \E_{\vec \xi}\|\vec \xi\|^2_n \ge \delta^4, 
$$
and because the noise is isotropic we immediately obtain 
$$V(\erm) = \E_{\vec \xi}[\|\widehat{f}_n(\vec \xi) - f_e\|^2_n] \ge \delta^4,$$
as desired.
\subsection{Proof of Theorem \ref{T:Rnd}}
\paragraph{Preliminaries}
The main tool we use from the theory of empirical processes is Talagrand's inequality \cite[Theorem 2.6]{koltchinskii2011oracle}:
\begin{lemma}\label{Lem:Talagrand}
Let $\cH$ be a class of functions on a domain $\cZ$ all of which are uniformly bounded by $M$. Let $Z_1,\ldots,Z_n \underset{i.i.d.}\sim \PP$. Then, there exist universal constants $ C,c > 0 $ such that
\begin{align*}
\Pr( |\|\cH\|_{n}  - \E \|\cH\|_{n}| \geq t) \leq C\exp\left(-\frac{cnt}{M} \log\left(1 + \frac{t}{\E \|\cH^2\|_{n}}\right)\right),
\end{align*}
where $
   \|\cH\|_{n}:= \sup_{h \in \cH} n^{-1} \sum_{i=1}^{n} h(Z_i),
$ 
and $\|\cH^2\|_{n} = \sup_{h \in \cH} n^{-1} \sum_{i=1}^{n} h(Z_i)^2$.
\end{lemma}
\begin{lemma}\label{Lem:DumTech}
    Under Assumptions \ref{A:Convex},\ref{A:Diameter},\ref{A:boudn}, the following holds:
    \[
        \nms \gtrsim \log(n)/n,
    \]
    where $\nms$ is defined in \eqref{Eq:minimax} above.
\end{lemma}
\begin{proof}[Proof of Lemma \ref{Lem:DumTech}]
Note that by Assumptions \ref{A:Convex},\ref{A:Diameter}, we have that
\[
    \log \cN(\eps,\cF,\PP) \gtrsim \log(\eps^{-1}),
\]
Therefore, we obtain that $\nm$ is greater or equal to the stationary point of the following equation:
\[
    \log(\eps^{-1})/n \asymp \eps^2,
\]
so $\nm \gtrsim \sqrt{\log n/n}$.
\end{proof}

\paragraph{Proof of Theorem \ref{T:Rnd}:}
We abbreviate $\cI_{L}:= \cI_{L}(n)$, $\cI_{U}:= \cI_{U}(n)$ and note that by the last lemma, we may assume that $\nms > C\log(n)/n$ for sufficiently large $C \geq 0$. For every fixed $\vec X, \vec \xi$, the function $\Lhat: \cF \to \R$ defined by
\begin{equation}\label{E:Erm}
    \Lhat(f) := \| \vec Y - f \|_n^2 - \|\vec \xi\|_n^2 = -2\Xf{f-f^{*}} + \| f- f^{*}\|_{n}^2,
\end{equation}
satisfies $\erm = \argmin_{f \in \cF} \Lhat(f)$. (Of course, $\Lhat(f)$ is just the empirical loss of $f$, up to subtracting a constant.) Note that $\Lhat(\ft) = 0$, so $\Lhat(\erm)$ must be non-positive.

Let $\{f_1, \ldots, f_{N}\}$ be an $\eps_{U}$-net of $\cF$ with respect to $\PP$ of cardinality $N = \cN(\eps_{U},\cF,\PP)$; for each $i \in [N]$, let $B(f_i):= B(f_i,\eps_{U})$ denote the ball of radius $\eps_{U}$ around $f_i$, so that the $B(f_i)$ cover $\cF$. For each $i$, let $L_{i}$ denote the minimal loss on the ball $B(f_i)$:  
\begin{equation}
        L_{i} :=  \min_{f \in B(f_i)} \Lhat(f)
\end{equation}

The main technical result is the following lemma: 
\begin{lemma}\label{Lem:events} Fix $i_* \in [N]$. For any absolute constant  $A > 0$, there exist absolute constants $C_1, C, > 0$, such that the following holds with probability of at least $1-2\exp(-A n \eps_{U}^4/\max\{\cI_{U},\nms\})$:
\begin{equation}\label{Eq:cballs}
\begin{aligned}
   &\forall i \in [N],\ \ \abs*{(L_i - L_{i_*}) - \E (L_{i} - L_{i_*})} \leq C_1\eps_{U}^2 + \frac{1}{4} \|f_i - f_{i_*}\|^2.
\end{aligned}
\end{equation}
\end{lemma}
We defer the proof of Lemma \ref{Lem:events} to the end of the section, and show how it implies the theorem.

\begin{proof}[Proof of Theorem \ref{T:Rnd} (assuming Lemma \ref{Lem:events})]
We apply Lemma \ref{Lem:events} with $i_* = \argmin_{i \in [N]} \E L_i$. 
Let $\cE$ denote the event of Lemma \ref{Lem:events} (the constant $A > 0$ in the lemma will be chosen shortly), and let $\cE'$ be the event that
\begin{equation}\label{Eq:wli}
\|f - g\|_n^2 \ge \frac{1}{2}\|f - g\|^2 - \cI_{L}
\end{equation}
for all $f, g \in \mathcal F$. By the definition of $\cI_{L}$,  $\cE'$ holds with probability $1 - n^{-1}$; in addition, a mildly tedious computation, which we defer to Lemma \ref{L:SmallDebt2}, shows that $A$ can be chosen such that $\Pr(\cE) \ge 1 - n^{-1}$ as well. In the remainder of the proof, we work on $\cE \cap \cE'$. 

Let $\widehat{f}_{i^{*}} = \argmin_{f \in B(f_{i_*})} \Lhat(f)$, so that $L_{i_*} = \Lhat(\widehat{f}_{i^{*}})$. Consider the function $h = \frac{\widehat{f}_{i^{*}} + \erm}{2}$, which lies in $\cF$ as $\cF$ is convex. We have
\begin{align}
\Lhat(h) &= -2 G_{h - \ft} + \|h - \ft\|_n^2 \nonumber \\
&= \frac{\Lhat(\widehat{f}_{i^{*}}) + \Lhat(\erm)}{2} + \left(\|h - \ft\|_n^2 - \frac{\|\widehat{f}_{i^{*}} - \ft\|_n^2 + \|\erm - \ft\|_n^2}{2}\right).
\label{Eq:loss_avg}
\end{align}
Applying the parallelogram law 
\[\|a + b\|_n^2 + \|a - b\|_n^2 = 2\|a\|_n^2 + 2\|b\|_n^2\] 
with $a = \frac{\widehat{f}_{i^{*}} - \ft}{2}$, $b = \frac{\erm - \ft}{2}$ yields
$$\|h - \ft\|_n^2 - \frac{\|\widehat{f}_{i^{*}} - \ft\|_n^2 + \|\erm - \ft\|_n^2}{2} = -\left\|\frac{\widehat{f}_{i^{*}} - \erm}{2}\right\|^2_n.$$
Combining this equation with \eqref{Eq:loss_avg} yields
\begin{equation*}
\Lhat(h) \le \frac{\Lhat(\widehat{f}_{i^{*}}) + \Lhat(\erm)}{2} - \left\|\frac{\widehat{f}_{i^{*}} - \erm}{2}\right\|^2_n.
\end{equation*}

But we also know that $\Lhat(h) \ge \Lhat(\erm)$ by the definition of $\erm$, so rearranging we obtain
\begin{equation}
\label{Eq:loss_bd}
\Lhat(\erm) \le \Lhat(\widehat{f}_{i^{*}}) - \frac{1}{2}\|\widehat{f}_{i^{*}} - \erm\|^2_n.
\end{equation}
Now let $i \in N$ such that $\erm \in B(f_i)$ and substitute $L_i = \Lhat(\erm)$, giving 
$$L_i \le L_{i_*} - \frac{1}{2}\|\widehat{f}_{i^{*}} - \erm\|^2_n.$$
Since we are on $\cE$, we may apply \eqref{Eq:cballs} and obtain 
$$\E L_i \le \E L_{i_*}  + C_3\eps_{U}^2 - \frac{1}{4}\|\widehat{f}_{i^{*}} - \erm\|^2_n \leq \E L_{i_*} + C_4\eps_{U}^2 + \cI_{L}- \frac{1}{8}\|\widehat{f}_{i^{*}} - \erm\|^2 .$$
But $\E L_{i_*} \le \E L_i$ by our choice of $i_*$, which implies finally that
$\|\widehat{f}_{i^{*}} - \erm\|_n^2 \lesssim \max \{\eps_U^2, \cI_L\} = \eps_{V}^2$.

Recall that we are also interested in $f \in \cO_{\delta_n}$ for $\delta_n = O(\eps_{V}^2)$. By the geometric argument in the proof of Theorem \ref{T:Fix}, for such $f$, we have $\|f - \widehat{f}_n\|_n = O(\delta_n)$.

Applying the lower isometry property \eqref{Eq:wli}, we obtain 
$$\|\erm - \widehat{f}_{i^{*}}\|, \|f - \erm\| \le C \eps_{V}$$ 
for any $f \in \cO_{\delta_n}$ on $\cE \cap \cE'$. Since $\|\widehat{f}_{i^{*}} - f_{i_*}\| \le \eps_{V}$ (as $\widehat{f}_{i^{*}} \in B(f_{i_*})$ by definition), we also have 
$$\|\erm - f_{i_*}\|, \|f - f_{i_*}\| \le C \eps_{V}.$$

In sum, thus far we have shown that under $\cE \cap \cE'$, an event of probability at least $1 - 2 n^{-1}$ any $f \in \cO_{\delta_n}$ satisfies $\|f - f_{i_*}\| \lesssim \eps_V$. It remains to show that this implies that $\|f - \E \erm\| \lesssim \eps_V$, for which it suffices to show that $\|\E \erm - f_{i_*}\| \lesssim \eps_V$. But 
$$\E\erm - f_{i_*} = (\E[\erm | \cE \cap \cE'] - f_{i_*}) \Pr(\cE \cap \cE') + (\E[\erm |(\cE \cap \cE')^{c}] - f_{i_*})\Pr((\cE \cap \cE')^{c}).$$
By what we have shown, $\|\E[\erm | \cE \cap \cE'] - f_{i_*}\| \le C \eps_V$, while $\|\E[\erm |(\cE \cap \cE')^{c}] - f_{i_*}\| = O(1)$ by Assumption \ref{A:boudn} and so the norm of the second term is asymptotically bounded by $O(n^{-1}) \ll \nm \le \eps_V$ because $\nm \gtrsim \sqrt{\frac{\log n}{n}}$ by Lemma \ref{Lem:DumTech}. This concludes the proof.
\end{proof}

It remains to prove the deferred lemmas. We begin with the most substantial one, Lemma \ref{Lem:events}. 

\begin{proof}[Proof of Lemma \ref{Lem:events}]
Recall that
\[
-L_i =\sup_{f \in B(f_i)} (2G_{f - \ft} - \|f - \ft\|_n^2).
\]
We write $f - \ft = (f - f_i) + (f_i - \ft)$, expand, and decompose this expression into terms depending on $f - f_i$ and terms depending only on $f_i - \ft$:

\begin{equation}
-L_i = A_i+A_i', \label{Eq:loss_decomp}
\end{equation}
where 
\begin{align*}
    A_{i} &:= \sup_{f \in B(f_i)} \left(2\Xf{f-f_i} - \|f - f_i\|_n^2 - 2 \langle f-f_{i}, f_i - \ft\rangle_n \right), 
    \\ A_i' &:=  2\Xf{f_i - \ft} - \| f_i - \ft \|_{n}^2.
\end{align*}
We also write
\begin{align*}
    B_i := A_i' - A_{i_*}' &=  2\Xf{f_i - \ft} - \| f_i - \ft \|_{n}^2 - (2\Xf{f_{i_*} - \ft} - \| f_{i_*} - \ft \|_{n}^2) 
    \\&= 2\Xf{f_i - f_{i_*}} - \|f_i\|_{n}^2 + \|f_{i_*}\|_{n}^2 + 2 \langle f_{i}- f_{i_*},\ft\rangle 
    \\&= 2\Xf{f_i - f_{i_*}} + \|f_{i} - f_{i_*}\|_{n}^2 - 2\|f_{i}\|_{n}^2 +2 \langle f_i,f_{i^*} \rangle+  2 \langle f_{i}- f_{i_*},\ft\rangle 
    \\&=2\Xf{f_i - f_{i_*}} + \|f_i - f_{i_*}\|_n^2 + 2\langle f_i - f_{i_*}, \ft-f_{i} \rangle
\end{align*}

We claim that with probability $1 - C\exp(-cn\eps_{U}^4/\max\{\nms,\cI_{U}\})$ the following holds: 
\begin{align}
\forall i \in [N], \quad &|A_i - \E A_i| \le C_1\eps_{U}^2, \label{Eq:ai_conc} \\
\forall i \in [N], \quad &|B_i - \E B_i| \le C_2\eps_{U}^2 + \frac{1}{4} \|f_i - f_{i_*}\|^2. \label{Eq:bi_conc}
\end{align}

Since $L_i - L_{i_*} = A_{i_*} - A_i - B_i$, combining \eqref{Eq:ai_conc} and \eqref{Eq:bi_conc} yields the lemma. 

We first prove \eqref{Eq:ai_conc}. For each $i \in [N]$, we control fluctuations of $A_i$ by applying Talagrand's inequality. To this end, write
$$A_i = \sup_{f \in B(f_i)} \frac{1}{n}\sum_{j = 1}^n a_f(X_j, \xi_j)$$
where 
$$a_f(x, \xi) = 2 \xi (f(x) - f_i(x)) - 2 (f(x) - f_i(x))(f_i(x) - \ft(x)) - (f(x) - f_i(x))^2.$$
To apply Talagrand's inequality, we need to bound $\E \sup_{f \in B(f_i)} n^{-1} \sum_{j = 1}^n a_f(X_j, \xi_j)^2$.

Using the identity $(a+b+c)^2 \leq 3(a^2+b^2+c^2)$, we see that
\begin{multline*}
\mathbb E_{\vec X, \vec \xi} \sup_{f \in B(f_i)}\int a_f(X_j, \xi_j)^2d\PP_n   \\ \le 3 \cdot \E_{\vec X, \vec \xi} \sup_{f \in B(f_i)} \int \left(2\xi^2 (f(x) - f_i(x))^2 + 2(f(x) - f_i(x))^2 (f_i(x) - \ft(x))^2 + (f(x) - f_i(x))^4\right)d\PP_n.
\end{multline*}
Using the assumptions $|\xi_i| \leq \Gamma_1$, $\|f\|_{\infty} \leq \Gamma_2$, where $\Gamma_1,\Gamma_2 > 0$ are some absolute constants, one can obtain that 
\[
\mathbb E_{\vec X, \vec \xi} \sup_{f \in B(f_i)}\int a_f(X_j, \xi_j)^2d\PP_n \leq C\cdot \mathbb E_{\vec X} \sup_{f \in B(f_i)}\int(f-f_i)^2d\PP_n  
\]
for $C \leq 12\max\{\Gamma_1^2,\Gamma_2^2\}$. Using the definition of the upper isometry constant $\cI_{U}$ and the stationary point $\eps_U$, we obtain
\[
    \mathbb E_{\vec X, \vec \xi} \sup_{f \in B(f_i)}\int a_f(X_j, \xi_j)^2d\PP_n \lesssim \max\{\cI_{U},\eps_{U}^2\} \asymp \max\{\cI_{U},\nms\},
\]
where the last step uses Lemma \ref{L:SmallDebt} below. 

Thus we may apply Talagrand's inequality to $A_i$ with  $\E\|\cH^2\|_n \lesssim  \max\{\nms,\eps_{U}^2\}$, giving
\begin{equation}
\begin{aligned}
      \Pr_{\vec X,\vec \xi} \crl*{|A_{i} - \E A_{i}| \geq t} &\leq C\exp(-cnt\log(1+t/\max\{\nms,\cI_{U}\})).  
\end{aligned}
\end{equation}
Taking a union bound over $i\in [N]$, we obtain
\[
 \Pr_{\vec \xi}\crl*{\exists i \in [N]: |A_{i} - \E A_{i}| \geq t}\leq C\exp(-cnt^2/\max\{\nms,\cI_{U}\}+\log N).
\]

Choosing $t = C_1\eps_{U}^2$ for $C_1$ sufficiently large and recalling that $\log N = \log \mathcal N(\epsilon, \cF, \PP) \le n \eps_U^4/\max(\cI_U, \nms)$ by \eqref{Eq:StatUI}, we obtain that 
\[
\forall i \in [N], \ \ |A_{i} - \E A_{i}| \leq C_1\eps_{U}^2
\]
with probability at least $1-2\exp(-c_1 n\eps_{U}^4/\max\{\nms,\cI_{U}\})$, which is \eqref{Eq:ai_conc}. 

Next, we handle $B_i$ for every $i \in [N]$. As in the case of $A_i$, we may write
$
B_i = n^{-1}\sum_{j = 1}^n b_i(X_i, \xi_i)
$
where 
\[
b_i(x, \xi) = 2\xi (f_i(x) - f_{i_*}(x)) + (f_i(x) - f_{i_*}(x))^2 + 2(f_i(x) - f_{i_*}(x))(\ft(x)-f_{i}(x)).
\]
We have $|b_i(x, \xi)| \le C |f_i(x) - f_{i_*}(x)|$, so as before,
\[\frac{1}{n} \sum_{j = 1}^n \E[b_i(X_j, \xi_j)^2] \le C \E[\|f_i - f_{i_*}\|_n^2] = C \|f_i - f_{i_*}\|^2,
\]
and hence by Bernstein's inequality, 
\[
\Pr( |B_{i} - \E B_{i}| \geq t) \le \exp\left(-\frac{c nt^2}{C_3 t + \|f_i - f_{i_*}\|^2}\right).
\]
 Substituting $t = t_i := C\eps_{U}^2 + \|f_i - f_{i_*}\|^2/4$, we obtain 
\begin{equation}
\begin{aligned}
        \Pr\left( |B_{i} - \E B_{i}| \geq C_2\eps_{U}^2 + \frac{\|f_i - f_{i_*}\|^2}{4}\right)  &\leq 2\exp\left(\frac{-c_1nt_i^2}{C_3 t_i +\| f_i - f_{i_{*}}\|^2}\right) \\&\leq 2\exp\left(-c_2n\max\left\{C\eps_{U}^2,\frac{\|f_i - f_{i_*}\|^2}{4}\right\}\right)
    \\&\leq 2\exp(-c_3 n \cdot C \eps_{U}^2).
\end{aligned}
\end{equation}
By the same exact argument as in the case of $A_i$, we may choose $C > 0$ sufficiently large such that with probability $1 - C \exp(-c n \eps_U^4/\max(\cI_U, \nms))$, 
\[
|B_i - \E B_i| \le C_2\eps_{U}^2 + \|f_i - f_{i_*}\|^2/4
\]
for every $i \in [N]$, which is \eqref{Eq:bi_conc}. This concludes the proof of Lemma \ref{Lem:events}. 
\end{proof}

\begin{lemma}\label{L:SmallDebt}
The following holds:
    \begin{equation}\label{Eq:EpsUs}
    \max\{\eps_{U}^2, \cI_{U}\} \asymp \max\{\nms,\cI_{U}\}.
\end{equation}
\end{lemma}
\begin{proof}
    If $\cI_{U} \lesssim \nms$, then $\nms \asymp \eps_{U}^2$ by definition. If $\cI_{U} \gtrsim \nms$, assume to the contrary $\cI_{U} \ll \eps_{U}$; as we have $n\eps_U^4 \asymp \cI_{U}\log \cN(\eps_U,\cF,\PP)$, this implies
    \[
        \log \cN(\eps_U,\cF,\PP)/n \gg \eps_{U}^2.
    \]
    But this implies, by definition of $\nm$, that $\nm \ge \eps_U$, contradicting the definition of $\eps_U$ (Definition \ref{D:eps_U}).
\end{proof}

\begin{lemma}\label{L:SmallDebt2} 
For a sufficiently large absolute constant $A > 0$, one has $\exp(-An \epsilon^4_U/\max(\eps_U^2, \cI_U)) \le n^{-1}$.
\end{lemma}
\begin{proof}
    First, we show that $N = \mathcal N(\cF, \eps_U, \mathbb P) \gtrsim n^{1/4}/\log n$. Suppose to the contrary that $N \ll n^{1/4}/\log n$. Then 
    $$n \eps_U^4 \asymp \cI_U \cdot \log N \lesssim \log n,$$
    since $\cI_U$ is at most the squared diameter of $\cF_n$, which is $\Theta(1)$. This yields $\epsilon_U \lesssim (\log n/n)^{1/4}$. But $\mathcal N(\eps,\cF, \PP) \ge \frac{1}{\eps}$ because $\diam_{\PP}(\cF) = \Theta(1)$ and $\cF_n$ is convex, so we obtain $\mathcal N(\eps_U,\cF, \PP) \gtrsim n^{1/4}/\log n$, contradiction.

    To upper-bound $\exp(-An \epsilon^4_U/\max(\eps_U^2, \cI_U))$, we split into cases. If $\cI_U \lesssim \eps_U^2$ then $\epsilon_U \asymp \nm$ and we have $n \nms \gtrsim \log n$ by Lemma \ref{Lem:DumTech}, so $\exp(-An \epsilon^4_U/\max\{\eps^2_U, \cI_U\}) \le \exp(-A n\nms) \le n^{-1}$ for sufficiently large $A > 0$. 
    
    Otherwise, if $\cI_U \gg \eps_U^2$ we have $n \epsilon^4_U/\cI_U \gtrsim \log N$ by the definition of $\epsilon_U$, and since $N \gtrsim n^{1/5}$, we have $\log N \gg \log n$. Hence, by choosing $A > 0$ large enough we can ensure that $\exp(-An \epsilon^4_U/\cI_U) \le n^{-1}$ in this case as well.
\end{proof}

\subsection{Proof of Theorem \ref{T:RndLowerIso}}
 Assume for simplicity that $c_L=1$. We say $\vec \xi$ has a Gaussian Isoperimetric Profile (GIP) with respect to $\| \cdot \|_{n}$, if for any measurable set $A \subset \R^n$ such that $\Pr_{\vec \xi}(A) \geq 1/2$, we have that 
\begin{equation}\label{Eq:GCP}
\Pr_{\vec \xi}(A_t) \geq 1- 2\exp(-nt^2/2). 
\end{equation}
where $A_t = \{\vec \xi \in \R^n: \inf_{x\in A}\|x-\vec \xi\|_{n} \leq t \}$. It is not hard to verify that the GIP and LCP are equivalent (cf. \citep[Thm 3.1.30]{artstein2015asymptotic}).

The main observation is the following simple and useful lemma which leverages the power of isoperimetry:
\begin{lemma}\label{Lem:Isop}
    For any measurable $A \subset \R^n$ such that $\Pr_{\vec \xi}(A) \geq 2\exp(-nt^2/2)$, $\Pr_{\vec \xi}(A_{2t}) \geq 1-2\exp(-nt^2/2)$.
\end{lemma}
\begin{proof}
 Since $A_{2t} = (A_t)_t$, \eqref{Eq:GCP} implies that it's sufficient to show that $\Pr(A_t) \geq 1/2$, and indeed it suffices to show that $\Pr(A_{t + \epsilon}) \geq 1/2$ for any $\epsilon > 0$. Fix $\epsilon > 0$, and assume to the contrary that $\Pr(B) > 1/2$, where $B = \R^n \backslash A_{t + \epsilon}$. It's easy to see that 
$ 
 A \subset \R^n \backslash B_{t + \epsilon}.
$
Hence, using \eqref{Eq:GCP}, we obtain
 \[
    \Pr(\R^n \backslash A) \ge \Pr(B_{t + \epsilon}) \ge 1 - 2\exp(-n(t + \epsilon)^2/2),
 \]
 i.e., $\Pr(A) \le 2 \exp(-n(t + \epsilon)^2/2) < 2\exp(-nt^2/2)$, contradiction.
\end{proof}

Denote the event of Definition \ref{D:Brack} by $\cE$, and recall the definition of $\nms$ via $n\nms \asymp \log \cN(\nm,\cF,\PP)$. Letting $S$ be an $\nm$-net of $\cF$ of cardinality $\cN(\nm, \cF, \PP)$, the pigeonhole principle implies the existence of $f_c \in S$ such that  
\[
    \Pr_{\vec \xi}(\underbrace{\{\erm \in B(f_c,\nm)\} \cap \cE }_{A}|\vec X) \geq \frac{\Pr_{\vec \xi}(\cE)}{\cN(\nm,\cF,\PP)} \geq \frac{\exp(-c_2n\nms)}{\cN(\nm,\cF,\PP)} \geq \exp(-c_3n\nms).
\]

By isoperimetry, $\Pr_{\vec \xi}(A_{2t}) \ge 1 - 2\exp(-n t^2/2)$, where $t = M\nm/2$ and $M$ is chosen such that $(M/2)^2 \ge 2 C_3$; this fixes the value of the absolute constant $M$ used in \eqref{eq:a_brack}.

Applying \eqref{eq:a_brack} yields that if $\vec \xi \in A \subset \cE$ and $\|\vec \xi' - \vec \xi\| \le M \nm = 2t$, $\|\erm(\vec \xi) - \erm(\vec \xi')\|^2 \le \rho_S(\vec X, \ft)$ and so $\|\erm(\vec \xi') - f_c\| \le \nm + \sqrt{\rho_{S}(\vec X,\ft)}$. This implies
\[
    \Pr_{\vec \xi}(\{\erm \in B(f_c, \nm + \sqrt{\rho_{S}(\vec X,\ft)})\}|\vec X) \geq \Pr_{\vec \xi}(A_{2t}|\vec X) \geq 1-2\exp(-n t^2/2).
\]
To bound the variance of $\erm$, we use conditional expectation as in Theorem \ref{T:Rnd}. We have 
\begin{align*}
    V(\erm|\vec X) &\le \E\|\erm - f_c\|^2 \le (1-2\exp(-n t^2/2)) \cdot (\nm + \sqrt{\rho_{S}(\vec X,\ft)}))^2 + 2 C\exp(-n t^2/2),
\end{align*}
where we have used the fact that $\diam_{\PP}(\cF) = \Theta(1)$. Recalling that $\nms \gtrsim \log(n)/n$, and $t = \Theta(\nm)$, we have $\exp(-nt^2) = O(\nms)$ and hence the RHS is bounded by $C \max\{\nms, \rho_S(\vec X, \ft)\}$, as desired.

\subsection{Proof of Theorem \ref{T:RandomFull}}\label{ss:rndFull}
For simplicity, we assume that $c_X = c_L = c_I = O(1)$, We abbreviate $\rho_{\cO}:= \rho_{\cO}(n,\PP,\ft)$.

We shall use the joint metric on $\cX^n \times \R^n$ given by 
$$\Delta_n((\vec X_1, \vec \xi_1), (\vec X_2, \vec \xi_2)) := n^{-1/2} \cdot d_n(\vec X_1, \vec X_2) + \|\vec \xi_1 - \vec \xi_2\|_n.$$
As $(\vec X, d_n)$ and $(\vec \xi, \|\cdot\|_2)$ both satisfy Lipschitz concentration inequalities with parameter $\Theta(1)$, so does the product space $\cX^n \times \R^n$ with the usual product metric 
$$((\vec X_1, \vec \xi_1), (\vec X_2, \vec \xi_2)) \mapsto d_n(\vec X_1, \vec X_2) + \|\vec \xi_1 - \vec \xi_2\|_2,$$ 
and since $\Delta_n$ is obtained by scaling this metric by $n^{-1/2}$, we obtain that $(\cX^n \times \mathbb R^n, \Delta)$ satisfies an LCP condition with parameter $\Theta(n)$. 

Let $\cE_1$ be the event of Assumption $\ref{A:Interpolate}$, namely, the event that the ERM is almost interpolating, and let $\cE_2$ be the event that $\diam_{\PP}(\cO_{M'\nms}) \le \rho_{\cO}$. Since $\Pr(\cE_1) + \Pr(\cE_2) > 1 + \exp(-c_{I} n \nms)$, we have $\Pr(\cE_1 \cap \cE_2) > \exp(-c_{I}n \nms)$. 

Set $\cE_3 = \cE_1 \cap \cE_2$. Since $\Pr(\cE_3) \ge \exp(-c_{I}n \nms)$, the same pigeonhole principle argument used in the proofs of Theorems \ref{T:Fix} and \ref{T:RndLowerIso} shows that there exists an absolute constant $c_1 \in (0,c_{I})$ and $f_c \in \cF$ such that
\[
    \Pr_{\vec X,\vec \xi}(\cE_3 \cap \{\erm \in B(f_c,\nm)\}) \geq \exp(-c_1 n \nms).
\]
Denote this event by $\cE$ (in this case, it is better to think about it as a subset of $\cX^n \times \mathbb R^n$). By the same argument as in Theorem \ref{T:RndLowerIso}, $\tilde \cE := \cE_{C_1\nm}$ will be an event of probability $1 - \exp(-c_1 n \nms)$, where $A_r = \{(\vec X, \vec \xi) \in \cX^n \times \R^n: \Delta_n((\vec X, \vec \xi), A) \le r\}$ as above, and $C_1$ is an absolute constant depending on $c_1$ and the LCP parameter of $\Delta$. 

Thus, we would like to show that any $(\vec X, \vec \xi) \in \tilde \cE$ is not too far from $f_c$. More precisely, we claim that for any $(\vec X, \vec \xi)$ at distance at most $C_1 \nm$ from $\cE$, the corresponding $\erm$ is at distance at most $C_2 \cdot \max\{\nm,\sqrt{\rho_{\cO}}\}$ from $f_c$.

For $f \in \cF$, let $f_{\vec X}$ denote $(f(X_1), \ldots, f(X_n)) \in \R^n$. We claim that it suffices to prove the following: for every $(\vec X_1,\vec \xi_1) \in \cE$ and such that $\Delta_n((\vec X_2,\vec \xi_2),(\vec X_1,\vec \xi_1)) \leq C_1\nm$, we have
\begin{equation}\label{Eq:endgameRND}
    \|\erm(\vec X_2,\vec \xi_2)_{\vec X_1} - \erm(\vec X_1,\vec \xi_1)_{\vec X_1}\|_n \lesssim \nm,
\end{equation}
where $\erm(\vec X_i,\vec \xi_i)$ is the ERM for the input points $\vec X_i$ and noise $\vec \xi_i$,. Indeed, assuming \eqref{Eq:endgameRND} we have

\begin{equation}\label{Eq:erm_perturb}
     \|\erm(\vec X_2,\vec \xi_2)_{\vec X_1} - \vec Y_1\|_n^2 \leq 2\|\erm(\vec X_1,\vec \xi_1)_{\vec X_1} - \vec Y_1\|_n^2 + 2\|\erm(\vec X_2,\vec \xi_2)_{\vec X_1} - \erm(\vec X_1,\vec \xi_1)_{\vec X_1}\|_n^2 \lesssim \nms,
\end{equation}
as the first term on the RHS is bounded by $2\nm$ because $(\vec X_1,\vec \xi_1) \in \cE$, and the second term is bounded by $4\nms$ by construction. We now specify the constant $M'$ in the definition of $\rho_{\cO}$ (Definition \ref{Def:rho_O}) to be any upper bound for the implicit absolute constant in \eqref{Eq:erm_perturb}. Under this definition, \eqref{Eq:erm_perturb} implies that $\erm(\vec X_2,\vec \xi_2)_{\vec X_1} \in \cO_{M'\nms}$ and hence $\|\erm(\vec X_2,\vec \xi_2) - \erm(\vec X_1,\vec \xi_1)\|^2 \le \rho_{\cO}$. Since $\erm(\vec X_1,\vec \xi_1) \in \cE$, this implies that 
$$\|\erm(\vec X_2,\vec \xi_2) - f_c\|^2 \le 2(\|\erm(\vec X_2,\vec \xi_2) - \erm(\vec X_1,\vec \xi_1)\|^2 + \|\erm(\vec X_2,\vec \xi_2) - f_c\|^2 \lesssim \max\{\nms, \rho_\cO\}$$
as desired.

Thus, on the high-probability event $\tilde \cE$, $\widehat{f}_n \in B(f_c, C \max\{\nms,\rho_{O}\})$. As in the proof of Theorem \ref{T:Rnd}, one concludes by conditional expectation that $V(\erm) \le C \max \{\nms,\rho_{O}\}$.

\paragraph{Proof of \eqref{Eq:endgameRND}:} 
For convenience, denote $f_{i, j} = \erm(\vec X_i, \vec \xi_i)_{\vec X_j}$, and similarly $\ft_j = (\ft)_{\vec X_j}$. As $d((\vec X_2,\vec \xi_2), (\vec X_1,\vec \xi_1)) \le 2\nm$, we have by the Lipschitz property that $\|f_{i, 1} - f_{i, 2}\|_n \le 2\nm$ and also $\|\ft_1 - \ft_2 \|_n \le 2\nm$. In addition, letting $\vec Y_i = \ft_i + \vec \xi_i$ be the observation vector, the Lipschitz property of $\ft$ and the bound on $\|\vec \xi_1 - \vec \xi_2\|_n$ together imply that $\|\vec Y_1 - \vec Y_2\|_n \le 4\nm$. The definition of $f_{i, i}$ as the ERM with data points $\vec X_i$ and observations $\vec Y_i$ implies that for $i = 1, 2$, 
\begin{equation}\label{eq:erm_minim}
\|f_{i, i} - \vec Y_i\|_n \le \|f_{\vec X_i} - \vec Y_i\|_n
\end{equation}
for any $f \in \cF$. Finally, the almost interpolating assumption (Assumption \ref{A:Interpolate}) yields $\|\vec Y_1 - f_{1, 1}\|_n \le C \nm$. 

We obtain \eqref{Eq:endgameRND} by putting these bounds all together. Indeed, we have
\begin{align*}
\|f_{1, 1} - f_{2, 1}\|_n &\le \|f_{1, 1} - \vec Y_1\|_n + \|\vec Y_1 - \vec Y_2\|_n + \|\vec Y_2 - f_{2, 2}\|_n + \|f_{2, 2} - f_{2, 1}\|_n \\
&\le (C + 6)\nm + \|f_{2, 2} - \vec Y_2\|_n,
\end{align*}
and substituting $i = 2$, $f = f_1$ into \eqref{eq:erm_minim} yields 
\begin{align*}
\|f_{2, 2} - \vec Y_2\|_n &\le \|f_{1, 2} - \vec Y_2\|_n \\
&\le \|f_{1, 2} - f_{1, 1}\|_n + \|f_{1, 1} - \vec Y_1\|_n + \|\vec Y_1 + \vec Y_2\|_n \\
&\le (C + 6)\nm,
\end{align*}
so we finally obtain
$$\|f_{1, 1} - f_{2, 1}\|_n \le 2(C + 6)\nm \lesssim \nm,$$
as desired.

\subsection{Proof of Theorem \ref{T:RndA}}\label{S:RndA}
The proof strategy is identical to that of Corollary \ref{C:A}: use a fixed-point theorem to find a function $\ft$ for which $\ft = \E_{\vec X, \vec \xi} \erm$, for which we have $\E\|\erm - \ft\|^2  \le \sup_{\ft \in \cF} V(\erm)$. However, the infinite-dimensional random-design setting makes things a bit trickier.

For given $\ft, \vec X, \vec \xi$, let $F_{\vec X, \vec \xi}(\ft)$ denote the corresponding ERM (which we have previously denoted $\erm$). Recall that while the ERM is uniquely defined as a vector in $\cF_n$, its  lift to $\cF$ is in general far from unique. We will make two temporary assumptions to streamline the proof, and explain at the end of the proof how to remove them, at the cost of some additional technical complexity. First, we assume that $F_{\vec X, \vec \xi}(\ft)$ is the (unique) element of $\cF$ of minimal $L^2(\PP)$-norm mapping to the finite-dimensional ERM; second, we assume that for each $\vec X$, the minimal-norm lifting map, defined by
$$L_{\vec X}(v) = \argmin \{\|f\|: f \in \cF\,|\, v = (f(x_1), \ldots, f(x_n))\},$$
is continuous.

The map $F_{\vec X, \vec \xi}$ is the composition of the following maps:
\begin{equation*}
\begin{tikzcd}
\cF \arrow[r,"P_n"]  &
\cF_n \arrow[r,"v \mapsto P_{\cF_n}(v + \vec \xi)"] &[5em]
\cF_n \arrow[r, "L_{\vec X}"] &
\cF
\end{tikzcd}
\end{equation*}

where $P_n(f) = (f(x_1), \ldots, f(x_n))$, $P_{\cF_n}$ is the projection from $L^2(\PPf)$ onto the convex set $\cF_n$, which is the LSE in fixed design, and $L_{\vec X}$ is the lifting map defined above. The linear map $P_n$ is continuous by Assumption \ref{A:Eval}, and the map $v \mapsto P_{\cF_n}(v + \vec \xi)$ is continuous because projection onto a convex set is continuous. As we have assumed (for now) that $L_{\vec X}$ is continuous, this proves that for every $\vec X, \vec \xi$, $f \mapsto F_{\vec X, \vec \xi}(f)$ is a continuous map of the compact set $\cF$ to itself. 

We claim that the expectation of this map, $f \mapsto \E_{\vec X, \vec \xi}[F_{\vec X, \vec \xi}(f)]$, is also continuous: indeed, if $f_k \to f$ then  
$$\|F_{\vec X, \vec \xi}(f_k) - F_{\vec X, \vec \xi}(f)\| \to 0$$
for each $\vec X, \vec \xi$ and is bounded by the diameter $D_{\PP}(\cF)$, so Jensen's inequality and dominated convergence imply
\begin{align*}
    \|\E[F_{\vec X, \vec \xi}(f_k)] - \E[F_{\vec X, \vec \xi}(f)]\| &\le \E[\|F_{\vec X, \vec \xi}(f_k) - F_{\vec X, \vec \xi}(f)\|] \to 0
\end{align*}
which is continuity.

We can thus apply the Schauder fixed point theorem \cite[Theorem 17.56]{aliprantis06infinite}: 

\begin{theorem}Let $K$ be a nonempty compact convex subset of a Banach space, and let $f: K \to K$ be a continuous function. Then the set of fixed points of $f$ is compact and nonempty.
\end{theorem}

The fixed point we obtain is a function $\ft \in \cF$ for which $\ft = \E[F_{\vec X, \vec \xi}(f)] = \E \erm$ and hence,
$$\E\|\erm - \ft\|^2 = \E \|\erm - \E\erm\|^2 \lesssim \cV(\erm,\cF,\PP).$$

This concludes the proof in the case that the lifting maps $L_{\vec X}$ are continuous.

Unfortunately, the assumption that the $L_{\vec X}$ are continuous turns out to be unjustified in general. Indeed, it is not difficult to construct an example of a convex set $K \subset \R^3$ for which the minimal-norm lift $P_{\R^2}(K) \to K$ is not continuous; in fact, one can construct $K \subset \R^3$ with no continuous section $P_{\R^2}(K) \to K$. So we need to explain how to proceed without this assumption.

Fortunately, each $L_{\vec X}$ is always continuous on the \textit{relative interior} of $\cF_n$ (we sketch the proof of this at the end of the section), so the following modification of $F_{\vec X, \vec \xi}$ does turn out to be continuous:

\begin{equation}\label{Eq:modERM}
\begin{tikzcd}
\cF \arrow[r,"P_n"]  &
\cF_n \arrow[r,"v \mapsto P_{\cF_n}(v + \vec \xi)"] &[5em]
\cF_n \arrow[r, "\varphi_\delta"] &
\cF_n \arrow[r, "L_{\vec X}"] &
\cF,
\end{tikzcd}
\end{equation}
where 
$$\varphi_{\delta}(v) = (1 - \delta) (v - v_0) + v_0$$
is simply a contraction of $\cF_n$ into a $(1 - \delta)$-scale copy of itself ($v_0$ is some arbitrarily chosen point in the interior of $\cF_n$).

Let $\tilde F_{\vec X, \vec \xi}$ denote the composition of the maps in \eqref{Eq:modERM}. By the argument above, $\tilde F_{\vec X, \vec \xi}$ is continuous and $\E[\tilde F_{\vec X, \vec \xi}]$ has a fixed point $\ft$.

Of course, $\ft$ is not a fixed point of $\E[F_{\vec X, \vec \xi}]$ as we would like. However, note that $\|\varphi_\delta(v) - v\|_n \le 2 \delta$ for any $v \in \cF_n$ (as the diameter of $\cF_n$ is at most $2$). Hence,  we have for any $v \in \cF_n$ that
$$\|L_{\vec X}(\varphi_\delta(v)) - L_{\vec X}(v)\|^2 \le 2\|\varphi_\delta(v) - v\|_n^2 + C\cI_{L}(n) \le 8\delta^2 + \cI_{L}(n)$$
on an event $\cE$ of high probability; in particular this holds for $v = P_{\cF_n}(P_n(\ft) + \vec \xi)$, which means that on $\cE$,
$$\|\tilde F_{\vec X, \vec \xi}(\ft) - F_{\vec X, \vec \xi}(\ft)\| \le 8\delta^2 + C\cI_{L}(n).$$
Choosing $\delta \lesssim \cI_{L}(n)$ and applying conditional expectation (using the fact that $\cE^c$ is negligible) and Jensen's inequality, we get that the $\ft$ thus obtained satisfies $\|\ft - \E \widehat{f}_n\| \lesssim \max\{\sup_{\ft \in \cF}V(\erm),\cI_{L}(n)\}$, which shows that the ERM is admissible for this $\ft$.

By the same argument, we may discard the assumption that the ERM is computed by finding the element of $\cF$ of minimal norm mapping to the finite-dimensional ERM $\widehat{f}_n^{(fd)}$: indeed, under the event $\cE$, the set of functions in $\cF$ mapping to $\widehat{f}_n^{(fd)}$ has diameter $C \cdot \cI_{L}(n)$, so changing the selection rule for the ERM will shift its expectation by a perturbation of norm at most $C \cdot \cI_{L}(n)$.

It remains to explain why the lifting map $L = L_{\vec X}: \cF_n \to \cF$ is continuous on the relative interior of $\cF_n$. Replacing the ambient space with the affine hull of $\cF_n$, we may assume $\cF_n$ has nonempty interior.

Suppose $v_k \to v$ in $\cF_n$ and $v \in \mathrm{int}\,\cF_n$; we wish to show that $L(v_k) \to f = L(v)$. As $\cF_n$ is compact, by passing to a subsequence we may assume $L(v_k)$ converges to some $g \in \cF$. Since $P_n$ is continuous, we have $v = P_n(L(v_k)) \to P_n(g)$, i.e., $g$ is a lift of $v$. Hence, by definition, $\|g\| \ge \|f\|$, and we wish to show that equality holds.

Suppose not. Then $\|g\| > \|f\|$ and hence $\|L(v_k)\| \ge \|f\| + \epsilon$ for all $k$ and some $\epsilon >  0$; that is, there exist $v_k$ arbitrarily close to $v$ whose minimal-norm lift has much larger norm than that of $v$. It suffices to show this is impossible (i.e., that $u \mapsto \|L(u)\|$ is upper semicontinuous at $v$). This follows from the fact that $u \mapsto \|L(u)\|$ is convex, as is easily verified, and a convex function is continuous on the interior of its domain \citep[Theorem 1.5.3]{schneider2014convex}; for completeness, we give a direct proof.

Since $v \in \mathrm{int}\,\cF_n$, there exists $r > 0$ such that $B(v, r) \subset \cF_n$. This implies that for any $\delta > 0$, one has
$$B(v, \delta) \subset v + \frac{\delta}{r} (\cF_n - v).$$

Let $D$ be the diameter of $\cF$ in $L^2(\PP)$. We have $\cF \subset B(f, D)$ and hence $\cF_n \subset P_n(B(f, D) \cap \cF)$. By linearity, this implies that
$$v + a(\cF_n - v) \subset P_n(B(f, a D) \cap \cF)$$
for any $a > 0$; choosing $a = \frac{\delta}{r}$ we obtain
$$B(v, \delta) \cap \cF_n \subset P_n\left(B\left(f, \frac{D\delta}{r}\right) \cap \cF\right).$$

In other words, if $\|u - v\|_n < \delta$ and $u \in \cF_n$, there exists an element of $\cF$ in $B(f, \frac{D\delta}{r})$ mapping to $u$, which in particular implies that $\|L(u)\| \le \|f\| + \frac{D\delta}{r}$. This means that $u \mapsto \|L(u)\|$ is upper semicontinuous at $v$, which was precisely what we needed in order to conclude that $L$ is continuous at $v$.

\subsection{Proof of Theorem \ref{C:Jagged}}\label{S:Jagged}
\paragraph{Preliminaries}
The following classical and standard results appear for example in \cite{vershynin2018high}.
\begin{lemma}\label{Lem:Maximal}[Maximal inequality]
    Let $Z_1,\ldots,Z_k$ be zero mean $\sigma$-sub-Gaussian random variables with bounded variance. Then, we have that
    \[
        \E \max_{1 \leq i \leq k} Z_i \lesssim \sigma \sqrt{\log k}.
    \]
\end{lemma}
\begin{lemma}[Dudley's lemma] \label{Lem:Dudley}
The following holds for all $\eps \in (0,1)$:
    \begin{equation}
    \E  \sup_{f_i \in \cN_{\eps}} \Xf{f_i - f^{*}} \leq \frac{C_4}{\sqrt{n}}\int_{\eps}^{\mathrm{Diam}_{\PPf}(\cF)}\sqrt{\log \cN(u,\cF,\PPf)} du,
\end{equation}
where $\cN_\epsilon$ denotes the  minimal $\epsilon$-net of $\mathcal F$ in terms of $L_2(\PPf)$. 
\end{lemma}
\begin{lemma}[Sudakov's minoration lemma] \label{Lem:Sudakov}
The following holds for all $\eps \in (0,1)$:
    \begin{equation}
    \E_{\vec \xi}  \sup_{f \in \cF} \Xf{f} \gtrsim \sup_{\eps \geq 0}\eps \sqrt{\frac{\log \cN(\eps,\cF,\PPf)} {n}}.
\end{equation}
\end{lemma}
\begin{proof}[Proof of Theorem \ref{C:Jagged}]
Throughout this proof, we fix a realization in $\PPf = \PP_n$ that satisfies $\cI_{L}(n) = o(\nms)$ and $\cI_{U}(n) = O(\nms)$. For such $\PPf$, note that 
\[
    \log \cN(\eps,\cF,\PP) \asymp \log \cN(\eps,\cF,\PPf) \quad \forall \eps \in (\nm ,\Gamma),
\]
we will use the last equation in various places in this proof.

We start by finding a weakly admissible $\ft$ by a more constructive method than that used in the proof of Corollary \ref{C:A} (the method here is closer to the original proof of \cite{chatterjee2014new}). Then, we   ove \eqref{Eq:admsta} for this choice of $\ft$.

Let $\cN:= \{f_1,\ldots,f_{\cN(\eps,\cF,\PPf)}\}$ be a minimal $\nm:=\eps_{*}(n)$-net of $\cF$ in terms of $L_2(\PPf)$, and denote $B(f_i):= B_n(f_i,\nm)$, $i=1,\ldots, |\cN|$. Our weakly-admissible $f^{*} \in \cF$ is defined as
\begin{equation}\label{Eq:ERMAD}
    f^{*}:= \argmax_{f_i \in \cN}\E \sup_{f \in B(f_i)}\Xf{f - f_i}.
\end{equation}

\begin{lemma}\label{Lem:AdEv}
The following event holds with probability (over $\vec \xi$) of at least $1-2\exp(-cn\nms)$ 
\begin{equation}\label{Eq:Ad1}
    \forall i \in [|\cN|] \quad \abs*{\sup_{f \in B(f_i)}\Xf{f-f_i} - \E \sup_{f \in B(f_i)}\Xf{f-f_i}} \leq C_1\nms.
\end{equation}
In addition, for a fixed $j \in 1,\ldots,\lceil\nm^{-1}\rceil$, the event
\begin{equation}\label{Eq:Ad2}
\begin{aligned}
    &\sup_{f_i \in \cN \cap B(f^{*}, j\nm)}\Xf{f_i - f^{*}} \leq  C_2 j\nms.
\end{aligned}
\end{equation}
holds with probability of at least $1 -2\exp(-c_3n\nms)$.
\end{lemma}
\noindent{The proof of this lemma appears below. We denote the event of \eqref{Eq:Ad1} by $\cE$, and by $\cE(j)$ the event of \eqref{Eq:Ad2}}. Following \cite{chatterjee2014new}, we define 
$$\Psi_{\vec \xi}(t) = \sup_{f \in B_{n}(f^{*},t)}2\Xf{f-f^{*}} - t^2.$$
One easily verifies (see \cite[Proof of Theorem 1.1]{chatterjee2014new}) that $\Psi_{\vec \xi}(t)$ is strictly concave and that $\argmax_{t \geq 0} \Psi_{\vec \xi}(t) = \|\ft - \erm\|_n^2$. 

This implies that if, for any particular $\vec \xi$, we identify $t_1, t_2$ such that $\Psi_{\vec \xi}(t_1) > \Psi_{\vec \xi}(t_2)$, the unique maximum of $\Psi_{\vec \xi}$ occurs for some $t$ smaller than $t_2$, i.e., $\|\ft - \erm\|_n \le t_2$. We will take $t_1 = \epsilon_*$ and $t_2 = D \epsilon_*$ for a sufficiently large constant $D$ and show that $\Psi_{\vec \xi}(t_1) > \Psi_{\vec \xi}(t_2)$ on $\mathcal E \cap \mathcal E(D)$. This implies that $\|\ft - \erm\|_n \le D\epsilon_*$ on $\mathcal E \cap \mathcal E(D)$, which precisely means that $\erm$ is admissible for $\ft$.

On the one hand, conditioned on $\cE$ we have
\begin{equation}\label{Eq:T1}
\begin{aligned}
    \Psi_{\vec \xi}(\nm) &=  \sup_{f \in B(\ft)}2\Xf{f-\ft} - \nms 
    \\&\geq \max_{f_i \in \cN}\E \sup_{f \in B(f_i)}2\Xf{f-\ft} - C_1\nms,
\end{aligned}
\end{equation}
where we used the definition of $\ft$ and Eq. \eqref{Eq:Ad1} above.
On the other hand, for $D \geq 2$ we have under $\cE(D) \cap \cE$ that
\begin{align*}
   \sup_{f \in B_n(\ft,D\nm)} \Xf{f-\ft} &\leq \max_{ f_i\in \cN \cap  B(\ft,D\nm)}\sup_{f  \in B(f_i)}\Xf{f-f_{i}}+ \sup_{ f_i\in \cN \cap  B(\ft,D\nm)}\Xf{f_i - \ft}
   \\&  \leq \sup_{f \in B(\ft,\nm)}\Xf{f-\ft} + C_2 D\nms,
\end{align*}
where we used the definition of $\ft$ and \eqref{Eq:Ad2}. Substituting in the definition of $\Psi$, we obtain
\begin{align*}
    \Psi_{\vec \xi}(D\nm) &= \sup_{f \in B_n(\ft,D\nm)}2\Xf{f-\ft} - D^2\nms \\
    & \leq 2\left(\sup_{f \in B(\ft)}\Xf{f-\ft} +C_2 D\nms\right) - D^2\nms \\
    &\leq  \Psi_{\vec \xi}(\nm) + (2 C_2+1)D\nms - D^2\nms.
\end{align*}
Comparing with \eqref{Eq:T1} we see that for $D \ge 2C_2 + C_1 + 1$ (say) we have $\Psi_{\vec \xi}(D\nm) < \Psi_{\vec \xi}(\nm)$ on $\mathcal E \cap \mathcal E(D)$. Since $\cI_{L}(n) = o(\nms)$,we obtain that 
\[
    \| \erm - \ft\|^2 \leq 4\|\erm - \ft\|_{n}^2 + o(\nms) \lesssim \nms,
\]
where we used that ,and therefore weakly admissible in $L_2(\PP)$.

Now, we are ready to prove \eqref{Eq:admsta}. First, we apply Sudakov's inequality, and note that
\begin{equation}\label{Eq:SudApp}
    \E_{\vec \xi}  \sup_{f \in \cF} \Xf{f - \ft} \gtrsim \sup_{\eps \geq 0}\eps \sqrt{\frac{\log \cN(\eps,\cF,\PPf)} {n}} \gtrsim \sup_{\eps \gtrsim \sqrt{\cI_{L}(n)}}\eps \sqrt{\frac{\log \cN(\eps,\cF,\PP)} {n}} \gg \nms,
\end{equation}
where we used that $\cI_{L}(n) = o(\nms)$ and that $\frac{\eps^2}{\log(1/\eps)} \cdot \log\cN(\eps,\cF,\PP)$ is decreasing in $\eps \in (0,\Gamma)$. 
 We first claim that with probability $1 - 2 \exp(-cn \nms)$,
\begin{equation}\label{Eq:GCC}
     \sup_{f \in B(f^{*})-\ft}\Xf{f} = \omega(\nms),
\end{equation}
where $\ft$ is our admissible function. 
 To see this, by Lemma \ref{Lem:Dudley} and \eqref{Eq:ERMAD}
\begin{equation*}\label{Eq:ScoreAd}
\begin{aligned}
     \omega(\nms) = \E \sup_{f \in \cF}\Xf{f - \ft} &\leq \max_{f_i \in \cN \cap B_n(\ft,\Ef)}\E \sup_{f\in B(f_i)}\Xf{f-f_i} + \E \max_{f_i\in \cN}\Xf{f_i-\ft} 
     \\&\leq \E \sup_{f \in B(f^{*})}\Xf{f-\ft} + \frac{C_1}{\sqrt{n}}\int_{\nm}^{\Gamma}\sqrt{\log \cN(t,\cF,\PP)}dt
     \\&\leq \E \sup_{f \in B(f^{*})}\Xf{f-\ft} + \frac{C_2}{\sqrt{n}}\int_{\nm}^{\Gamma}\sqrt{\log \cN(t,\cF,\PPf)}dt
     \\&\leq \E \sup_{f \in B(f^{*})}\Xf{f-\ft} +O(\nms),
\end{aligned}
\end{equation*}
where we used \eqref{Eq:SudApp} and Lemma \ref{Lem:Dudley} and that $\cI_{L}(n) = o(\nms)$ and $\cI_{U}(n) = \Theta(\nms)$. Hence, 
\begin{equation}\label{Eq:GCC_exp}
\E \sup_{f \in B(f^{*})}\Xf{f-\ft} = \omega(\nms) - O(\nms) = \omega(\nms).
\end{equation}
This gives us a lower bound for $\sup_{f \in B(f^{*})}\Xf{f-\ft}$ in expectation, and high-probability bound follows from the proof of Lemma \ref{Lem:AdEv} below; so we only sketch it: $\sup_{f \in B(\ft)} G_{f - \ft}$ is convex and $O(\epsilon_* n^{-1/2})$-Lipschitz, which means that it deviates from its expectation by $\epsilon_*^2$ with probability at most $2 \exp(-cn \nms)$. Combining this with \eqref{Eq:GCC_exp} proves \eqref{Eq:GCC}.

Let $\cV$ denote the set of noise vectors for which $\|\widehat{f}_n - \ft\|_n \le C \nm$ and $\sup_{f \in B(\ft)}\Xf{f- \ft} = \omega(\nms)$; by what we have already proven, we have 
$$\Pr(\vec{\vec \xi} \in \cV) \ge 1 - C \exp(-cn \nms).$$
Let $\cV' = \cV \cap (-\cV) = \{\vec{\vec \xi}: \vec \xi, -\vec \xi \in \cV\}$. By the union bound,
$$\Pr(\vec{\vec \xi} \in \cV') \ge 1 - 2C \exp(-cn \nms).$$
Fix any $\vec{\vec \xi} \in \cV'$, and denote by $\widehat{f}_n^-$ the ERM with the flipped noise vector $-\vec{\vec \xi}$: 
\[
\erm^{-} := \argmin_{f \in \cF} \left(\sum_i (-\xi_i + \ft(x_i) -f(x_i))^2\right).
\]
Since $\vec{\vec \xi}, -\vec{\vec \xi} \in \cV' \subset \cV$, we have 
$$\|\widehat{f}_n - \widehat{f}_n^-\|_n \le \|\widehat{f}_n - \ft\|_n + \|\widehat{f}_n^- - \ft\|_n  \le C \nm.$$

In other words, $\widehat{f}_n^- \in B_n(\widehat{f}_n, C\nm)$, so to prove \eqref{Eq:admsta}, it thus suffices to show that $\widehat{f}_n^-$ is not a $\delta$-approximate minimizer for $\delta = \omega(\nms)$ with respect to the noise $\vec{\vec \xi}$, i.e.,
$$\frac{1}{n}\sum_{i=1}^{n}(\ft(x_i) + \xi_i - \widehat{f}_n^-(x_i))^2  \geq  \frac{1}{n}\sum_{i=1}^{n}(\ft(x_i) + \xi_i - \erm(x_i))^2 + \omega(\nms).$$
Equivalently, (by subtracting $\|\vec \xi\|^2_n$ from both sides as in \eqref{E:Erm}), we wish to prove that
$$-2\Xf{\widehat{f}_n^- -f^{*}} + \|\widehat{f}_n^- - f^{*}\|_{n}^2 \ge -2\Xf{\erm -f^{*}} + \|\erm - f^{*}\|_{n}^2 + \omega(\nms).$$
Since $\|\widehat{f}_n^- - f^{*}\|_{n}^2, \|\widehat{f}_n^- - f^{*}\|_{n}^2 = O(\nms)$ as $\vec \xi, -\vec \xi \in \cV$, this reduces to showing that 
\begin{equation}\label{Eq:xf_xi_flip}
-2\Xf{\widehat{f}_n^- -f^{*}} \ge -2\Xf{\erm -f^{*}} + \omega(\nms).
\end{equation}

On the one hand, by using Eqs. \eqref{Eq:T1} and \eqref{Eq:GCC} above it is easy to see that on $\cV'\subset \cV$, we have  $\Xf{\erm -f^{*}}\ge \omega(\nms)$. On the other hand, 
\begin{equation}\label{Eq:xf_xi_flip2}
\Xf{\widehat{f}_n^- - \ft} = \frac{1}{n} \langle \erm - \ft, \vec \xi\rangle = -\frac{1}{n}\langle \widehat{f}_n^- - \ft, -\vec \xi\rangle.
\end{equation}
But note that $\frac{1}{n}\langle \widehat{f}_n^- - \ft, -\vec \xi\rangle$
is the  process for the noise vector $-\vec{\vec \xi}$ evaluated at the corresponding ERM, namely $\widehat{f}_n^-$, and we have $-\vec \xi \in \cV'$, which implies that 
$$\frac{1}{n}\langle \widehat{f}_n^- - \ft, -\vec \xi\rangle \ge \omega(\nms).$$ Combining these last two inequalities we see that \eqref{Eq:xf_xi_flip} indeed holds over all $\cV'$, which implies that \eqref{Eq:admsta} holds on $\cV'$, as desired.
\end{proof}

 It remains to prove Lemma \ref{Lem:AdEv}.
 
\begin{proof}[Proof of Lemma \ref{Lem:AdEv}]
First, define  
\[
F_i(\vec \xi):= 2 n^{-1}\sup_{f \in B(f_i)}\sum_{k=1}^{n}(f-f_i)(x_k)\cdot \xi_k.
\]
Since $\|f - f_i\|_n \le 2\|f - f_i\| + O(\nm) = O(\nm)$ for all $f \in B(f_i)$, we see that $F_i(\cdot)$ is a $O(\nm n^{-1/2})$-Lipschitz and also convex function (with respect to the usual Euclidean norm on $\R^n$). Hence, we apply \eqref{Eq:CCP} and obtain
\begin{equation}
    \Pr_{\vec \xi} \crl*{\left|\sup_{f \in B(f_i)} \Xf{f-f_i}-\E \sup_{f \in B_n(f_i)} \Xf{f-f_i}\right| \geq t} \leq 2\exp(-cnt^2/\nms).
\end{equation}
Therefore, by taking a union bound over $1 \leq i \leq |\cN|$
\[
 \Pr_{\vec \xi}\crl*{\forall i \in [|\cN|]: \left|\sup_{f \in B_n(f_i)} \Xf{f-f_i}-\E \sup_{f \in B_n(f_i)} \Xf{f-f_i}\right| \geq t} \leq 2\exp(-cnt^2/\nms+\log|\cN|).
\]
Now, recall the definition of $\cN := \cN(\nm,\cF,\PPf)$ and that by the definition of the minimax rate, 
\[
\log \cN /n \asymp  \nms. 
\] 
This allows us to choose $t = C\nms$ (for large enough $C > 0$) such that with probability of at least $1-2\exp(-cn\nms)$, the following holds: 
\begin{equation}
\forall i \in [|\cN|]:  \left|\sup_{f \in B_n(f_i)} \Xf{f-f_i}-\E \sup_{f \in B_n(f_i)} \Xf{f-f_i}\right| \leq C\nm \sqrt{\log |\cN|/n} \leq C_1\nms,
\end{equation}
which proves \eqref{Eq:Ad1}.

For the second part of the lemma, fix $j \geq 1$, and define
\[
    F_j(\vec \xi) = \sup_{f \in \cN \cap B_n(f^{*},j\nm)} 2\Xf{f_j - \ft}.
\]
Again, it is easy to verify that $ F_j(\cdot)$ is convex and $2j\nm n^{-1/2}$-Lipschitz.  Using \eqref{Eq:CCP} once again, we obtain that
\[
\Pr\crl*{\left|\sup_{f \in \cN \cap B_n(f^{*},j\nm)} \Xf{f-f^{*}} - \E  \sup_{f \in \cN \cap B_n(f^{*}, j\nm)} \Xf{f-f^{*}}\right| \geq t} \leq 2\exp(-cn(t/j)^2/\nms).\] 
Choosing $t \sim j\nm/\sqrt{n} \lesssim j\nms$, we obtain 
\begin{equation}\label{Eq:sup_balls_conc}
\Pr\crl*{\left|\sup_{f \in \cN \cap B_n(f^{*},j\nm)} \Xf{f-f^{*}} - \E  \sup_{f \in \cN \cap B_n(f^{*},j\nm)} \Xf{f-f^{*}}\right| \geq j \nms} \leq 2\exp(-c\nms).
\end{equation}

 Next, by applying the maximal inequality (Lemma \ref{Lem:Maximal}) over $|\cN|$ random variables, and the definition of the minimax rate,
 \begin{equation}\label{Eq:sup_balls_exp}
    \E  \sup_{f \in \cN \cap B_n(f^{*},j\nm)}\Xf{f-f^{*}} \leq Cj\nm\sqrt{\log |\cN|/n} \leq C_1j\nms. 
 \end{equation}
 Combining \eqref{Eq:sup_balls_conc} and \eqref{Eq:sup_balls_exp} yields \eqref{Eq:Ad2}, concluding the proof.
\end{proof}
\section{Loose Ends}\label{S:Loose}
\begin{lemma}\label{L:Stable}
    Under Assumption \ref{A:Convex} the following holds: $\rho_{S}(\vec X) \lesssim \max\{\cI_{L}(\vec X),\nms\}$ for any realization $\vec X$.
\end{lemma}
\begin{proof}
    It is always the case that $\rho_{S}(\vec X) \leq \max\{\cI_{L}(\vec X),\nms\}$ for any $\vec X$. Indeed, if  \eqref{eq:a_brack} holds, then for any estimator $\bar f_n$ such that $\vec \xi \mapsto \bar f_n(\vec \xi)$ is $1$-Lipschitz, $\|\vec \xi - \vec \xi'\|_{n} \lesssim \nm$ implies that
\[
    \|\bar{f}_n(\vec \xi') - \bar{f}_n(\vec \xi)\|^2 \leq 2(\|\bar{f}_n(\vec \xi') - \bar{f}_n(\vec \xi)\|_{n}^2 + \cI_{L}(\vec X)) \leq 2(\|\vec \xi' - \vec \xi\|_{n}^2 + \cI_{L}(\vec X)) \lesssim \max\{\nms,\cI_{L}(\vec X)\}
\]
deterministically, not just with non-negligible probability.
\end{proof}

\begin{lemma}\label{L:Brack}
Under Assumptions \ref{A:Convex},\ref{A:IsoData},\ref{A:Interpolate}, we have that 
$$\rho_{\vec S}(\vec X, \ft) \le \rho_{\cO}(n, \PP, \ft) \lesssim \max\{\cI_L(n, \PP),\nms\}.$$
\end{lemma}
\begin{proof}
Let $\cE \subset \cX^n \times \mathbb R^n$ be the event that $\diam_{\PP}(\cO_{M'\nms}) \le \rho_{\cO}(n, \PP, \ft)$ and $\|\widehat{f}_n(\vec X, \vec \xi) - \vec Y\|_n^2 \le C_I \nms$. For any $\vec X \in \cX^n$, let $\cE_{\vec X} = \{\vec \xi \in \mathbb R^n\,|\,(\vec X, \vec \xi) \in \cE\}$.

Since $\Pr(\cE) \ge \exp(-c_I n\nms)$ by definition, the set 
$$\cE' = \{\vec X \in \cX^n\,|\, \Pr_{\vec \xi}(\cE_{\vec X}) \le \rho_{\cO}(n, \PP, \ft)) \ge \exp(-c_I n \nms) \in \cE\}$$ 
satisfies $\Pr_{\cX^n}(\cE') \ge \exp(-(c_I/2) n \nms)$ by Fubini's theorem.

Fix $\vec X \in \cE'$, $\vec \xi \in \cE_{\vec X}$, and let $\vec Y =  \ft|_{\vec X} + \vec \xi$ as usual. If $\|\vec \xi' - \vec \xi\|_n \le \frac{\sqrt{M' - C_I}}{2}\nm$ then 
\begin{align*}\|\widehat{f}_n(\vec X,\vec \xi') - \vec Y\|_n^2 \le 2(\|\widehat{f}_n(\vec X, \vec \xi') - \widehat{f}_n(\vec X, \vec \xi)\|_n^2 + \|\widehat{f}_n(\vec X, \vec \xi) - \vec Y\|_n^2) \\ 
\le  2((M' - C_I) \nms + C_I) \nms = M'\nms
\end{align*} 
where we have used $\|\widehat{f}_n(\vec X, \vec \xi') - \widehat{f}_n(\vec \xi)\|_n \le \|\vec \xi - \vec \xi'\|_n$ as $\widehat{f}_n$ is $1$-Lipschitz in the noise. In particular $\widehat{f}_n(\vec X, \vec \xi') \in \cO_{M'\nms}$ and so $\|\widehat{f}_n(\vec X, \vec \xi') - \widehat{f}_n(\vec X, \vec \xi)\| \le \rho_{\cO}(n, \PP, \ft)$, as $(\vec X, \vec \xi) \in \cE$. Thus,  if $M'$ is chosen large enough so that $M \le \frac{\sqrt{M' - C_I}}{2}\nm$, one obtains \eqref{eq:a_brack} is satisfied with $\delta(n) = \rho_{\cO}(n, \PP, f)$ and $c_2 = c_I/2$, implying that $\rho_{\vec S}(\vec X, \ft) \le \rho_{\cO}(n, \PP, f)$.

To see that $\rho_{\cO}(n, \PP, \ft) \lesssim \max\{\cI_L(n, \PP), \nms\}$ is even easier: it's easy to see that $\diam_{\PP_n}(\cO_{M'\nms}) \lesssim \nm$ (see the end of the proof of Theorem \ref{T:Fix} for details), and the definition of the lower isometry remainder implies that $\diam_{\PP}(\cO_{M'\nms}) \lesssim \diam_{\PP}(\cO_{M'\nms}) + \sqrt{\cI_L(n, \PP)}$ on the high-probability event $\cI_L(\vec X) \le \cI_L(n, \PP)$. This yields that for an appropriate choice of $C > 0$, $\delta(n) = C \max\{\cI_L(n, \PP), \nms\}$ satisfies \eqref{eq:rho_O} and hence $\rho_{\cO}(n,\PP, \ft) \le \max\{\cI_L(n, \PP),\nms\}$.
\end{proof}
\subsection{Full Proof of Theorem  \ref{C:A} }\label{ss:rdchat}
Note that when the convex set 
\[
\cF_n:= \{(f(X_1),\ldots,f(X_n)): f \in \cF\}
\]
is not compact, we cannot apply the fixed point theorem directly. However, if we find $\ft \in \cF$ such that
\begin{equation}
      B^2(\erm) \leq \max\{C_1 \cdot V(\erm), C_2/n\},
\end{equation}
where $C_1,C_2 > 0$ are some absolute constants. Then, the proof follows from the argument of \S \ref{ss:pChat}, since the fixed point theorem was only used to find a $\ft \in \cF$ satisfying the last equation.

First, let us provide some intuition to our proof. The idea is to find $\ft \in \cF$ with low bias by an iterative argument similar to the proof of Banach's fixed point theorem. If $\erm$ has a ``high'' bias on some underlying $f_0 \in \cF$, then, $\erm$ should have a lower or equal bias when the underlying function is  $f_1 = \E_{0} \erm$, where $\E_{0}$ means taking expectation when $\ft = f_0$.  If ERM attains a ``low'' bias, then we are done. Otherwise, consider the underlying $f_2 = \E_{1}\erm$, and repeat this process for $n$ times. We will show that some $m = O(n)$, $\ft = f_{m}$ will be our ``admissible'' function.

This idea is captured in the following lemma (that we will prove below):
\begin{lemma}\label{Lem:iteration}
    Let $f_0 \in \cF$ and for any $i \geq 1$ denote by  $f_{i} = \E_{i-1} \erm$. Then, there exists $m = O(n)$ such that  
    \begin{equation}\label{Eq:SC}
        B_m^2(\erm) \leq 3 \cdot V_m(\erm) + C/n,
    \end{equation}
    where $B_m^2,V_m$ are the bias and the variance of ERM when $\ft = f_m$.
\end{lemma}
\begin{proof}[Proof of Lemma \ref{Lem:iteration}]
    First, note that 
    \begin{equation}\label{Eq:LSEMAX}
        \erm = \argmax_{f \in \cF} \{2\langle \vec \xi,f - \ft\rangle_n - \| f - \ft \|_{n}^2\}
    \end{equation}
    For each $(\vec \xi,\ft) \in \R^{2n}$ define the score function $L_{\vec \xi,\ft}:\R^n \to \R^{+}$
    \[
        L_{\vec \xi,\ft}(f) = 2\langle  \vec \xi, f - \ft\rangle_n - \|f - \ft \|_{n}^2.
    \]
    and let
  $L_{\ft}(\erm) := \Med  L_{\vec \xi,\ft}(\erm)$. Note that as $\vec \xi $ is isotropic Gaussian then \[ 1- O(1/\sqrt{n}) \leq L_{\ft}(\erm) \leq 2 + O(1/\sqrt{n}).\]   

    \paragraph{Claim.} For every $\ft \in \cF$, either

    \begin{equation}\label{Eq:LStop}
    B^2(\erm) \leq 3V(\erm) + C/n
    \end{equation}
    or
    \begin{equation}\label{Eq:LIter}
        L_{\E \erm}(\erm) -  \E L_{\ft}(\erm) \geq \frac{B^2(\erm)}{3},
    \end{equation}
where $L_{\E_0 \erm}(\erm) $ means that $\erm$ is with respect that the ground truth $\E_0\erm$, and $\E L_{\ft}(\erm)$  means that $\erm$ over the ground truth $\ft$.

    The lemma follows from the claim by an iterative argument. Indeed, let $f_0 = \ft \in \cF$ be some function. If $f_0$ satisfies \eqref{Eq:LStop}, we are done. Otherwise, $f_1 = \E \erm$ satisfies $L_{f_1}(\erm) - L_{f_0}(\erm) \ge \frac{C'}{n}$. if $f_1$ satisfies \eqref{Eq:LStop}, then we stop. Otherwise, repeat the same argument with $\ft = f_1$ and $f_2 = \E \erm(f_1)$, and so on. Since the score $L_f(\erm)$ is bounded above by a constant, eventually some $f_{m}$ with $m \le n/C$ will have to satisfy \eqref{Eq:SC}. It thus remains to prove the claim.

    \paragraph{Proof of the claim} Suppose that \eqref{Eq:LStop} does not hold, i.e., $B_0^2(\erm) \le 3 V_0(\erm) + C/n$. Write $f_0 = \ft$,  $f_1  = \E_0 \erm$, set $\tilde V_0(\erm):= \max\{3V_0(\erm),C/n\}$ and let $\tilde{f}_n$ be the restricted LS on 
    \[\cG:= B_n(f_1,\sqrt{\tilde{V}_0(\erm)}) = \{ f \in \cF: \|f - f_1\|_{n}^2 \leq \tilde{V}_0(\erm)\} ,\] namely, 
    $
        \tilde{f}_n := \argmin_{f \in \cG}\sum_{i=1}^{n}(Y_i - f_1(x_i))^2.
    $  Using \eqref{Eq:LSEMAX},  we have that
    \[
       L_{\vec \xi,f_1}(\erm) \geq L_{\vec \xi,f_1}(\tilde{f}_n) =  \sup_{f \in \cG} 2\langle  \vec \xi, f - f_1\rangle_n - \| f - f_1 \|_{n}^2  \geq \sup_{f \in \cG} 2\langle \vec \xi, f - f_1\rangle_n - \tilde{V}_0(\erm),
    \]
    where the first inequality follows from $\cG \subseteq \cF$, and the second equality follows from the fact that $\mathrm{Diam}(\cG)^2 \leq \tilde{V}_0(\erm)$.
    
   By Chebyshev's inequality, with probability at least $2/3$, we have that $\erm \in \cG$; we denote this event by $\cE$. Under this event, we have that 
   \[
     \sup_{f \in \cG} \langle \vec \xi, f \rangle_n \geq  \langle \vec \xi, \erm \rangle_n.
   \]
   Hence, using the last two equations, the following holds on the event $\cE$:
    \begin{align*}
        L_{\vec \xi,f_1}(\erm) - L_{\vec \xi, f_0}(\erm) &\geq L_{\vec \xi,f_1}(\tilde{f}_n) - L_{\vec \xi,f_0}(\erm) \\&\geq \|\erm - f_0\|_{n}^2 -  \langle \vec \xi, f_1 - f_0 \rangle_n -\tilde{V}_0(\erm).
    \end{align*}
    Next
    , note that $\langle  \vec \xi, f_1 - f_0 \rangle_n \sim N(0,\|f_1 - f_0\|_{n}^2) = N(0,B_0^2(\erm))$. Hence, there exists an event $\cE_1 \subset \cE$ that holds with probability of at least $0.6$ such that
    \begin{equation}\label{Eq:MedMed}
        L_{\vec \xi,f_1}(\erm) - L_{\vec \xi,f_0}(\erm) \geq  \|\erm - f_0\|_{n}^2 -  \tilde{V}_0(\erm) - C \cdot B(\erm)/ \sqrt{n}.
    \end{equation}
    Next, using the fact that $\| \erm - f_0 \|_{n}$ is  $n^{-1/2}$-Lipschitz and the LCP inequality \eqref{Eq:CCP}, we know that \[\| \erm - f_0 \|_{n} -\E_0\| \erm - f_0 \|_{n}\] is zero mean and $n^{-1/2}$ sub-Gaussian, so by a standard tail integration, one has
    \[
      \E_0\|\erm - f_0\|_{n}^2 = \Med_0\|\erm - f_0\|_{n}^2 + O(\max\{\E_0 \|\erm - f_0\|_{n}/\sqrt{n},1/n\}).
    \]
    Finally, we take a median over \eqref{Eq:MedMed}, and use the last equation and obtain:
    \begin{align*}
        L_{f_1}(\erm) - L_{f_0}(\erm)&\geq \E_0\|\erm - f_0\|_{n}^2 - 3V_0^2(\erm) - O(\max\{\sqrt{B_0^2(\erm)/n},\E_0 \|\erm - f_0\|_{n}/\sqrt{n},1/n\}) 
        \\& = B_0^2(\erm) + V_0^2(\erm) - 3V_0^2(\erm) - O(\max\{\E_0 \|\erm - f_0\|_{n}/\sqrt{n},1/n\})
        \\& = B_0^2(\erm) - 2V_0^2(\erm) - O(\max\{\E_0 \|\erm - f_0\|_{n}/\sqrt{n},1/n\})
        \\&\geq B_0^2(\erm)/3,
        \end{align*}
        where we used the assumption $B_0^2(\erm) \geq 3V_0(\erm) + C/n$, for $C \geq 0$  large enough; the claim follows.
\end{proof}
\begin{remark}[On the proof under compactness]
The Brouwer fixed point theorem, which we used in  \S \ref{ss:pChat} to obtain an existence of $\ft \in \cF$ for which $\E_{\vec \xi} \erm = \ft$, is a deep result, and one may ask whether it is essential to the proof. Another commonly used fixed-point theorem is that due to Banach; the Banach fixed point theorem is elementary, but requires a bound $\|F(f) - F(g)\|_n \le c\|f - g\|_n$ for some $c < 1$ and all $f, g \in \cF$. One has 
\begin{equation}\label{eq:proj_avg}
    \|F(f) - F(g)\|_n \le \E_{\vec \xi} \|\erm(f + \vec \xi) - \erm(g + \vec \xi)\|_n,
\end{equation} 
 Note that $\|\erm(f + \vec \xi) - \erm(g + \vec \xi)\|_n \le \|f - g\|_n$ because $P$ is $1$-Lipschitz. Also, it's easy to see that there exists some $\vec \xi$ for which $\|\erm(f + \vec \xi) - \erm(g + \vec \xi)\|_n$ is strictly smaller than $\|f - g\|_n$, and continuity of $P$ ensures that the same holds for all $\vec \xi'$ sufficiently close to $\vec \xi$, implying $\|F(f) - F(g)\|_n < \|f - g\|_n$. But this is not yet sufficient to apply the Banach fixed point theorem. 

Via more delicate convex-geometric arguments, though, one can show that if $\|\vec \xi\|_n$ is sufficiently large compared to the diameter of $\cF$ (say, $\|\vec \xi\|_n > C \cdot \diam(\cF)$) and $\langle f - g, \vec \xi\rangle_n\ge \epsilon \|f - g\|_n \|\vec \xi\|_n$ (i.e., the angle between $f - g$ and $\vec \xi$ is bounded away from 90 degrees) then 
$$\|\erm(f + \vec \xi) - \erm(g + \vec \xi)\|_n \le (1 - \delta)\|f - g\|_n$$ 
for some $\delta$ depending on $\epsilon$ and $C$, which allows one to conclude, using \eqref{eq:proj_avg}, that $\|F(f) - F(g)\|_n \le c \|f - g\|_n$ for some $c < 1$ and all $f, g \in \cF$. Hence, the Banach rather than the Brouwer fixed-point theorem can be used in the proof, rendering it elementary but more technical. 
\end{remark}

\subsection{Addendum to \S \ref{ss:ronT} }
\begin{example}\label{Ex:One}
Let $\cX = \mathbb{S}^n$, the Euclidean unit sphere of dimension $n$ contained in $\mathbb R^{n + 1}$, let $\PP$ be the uniform measure on $\cX$, and for each hyperplane $H$ passing through the origin, let $\PP^H$ be the uniform measure on $\cX \cap H \cong \mathbb{S}^{n - 1}$. The set of such hyperplanes is the real $(n + 1, n)$ Grassmanian, denoted $\Gr$. 

For each $H$, let $1_H$ denote the characteristic function of $H$, and let $\cF = \conv \{1_H, 1 - 1_H: H \in \Gr\}$. Note that in $L^2(\PP)$, $\cF$ reduces to the class of constant functions between $0$ and $1$ because for any $H$, $1 - 1_H \equiv 0$ and $1 - 1_H \equiv 1$ almost everywhere on $\cX$. Similarly, in $L^2(\PP^H)$, $1_H \equiv 1$ and $1 - 1_H \equiv 0$, while for any $H' \neq H$, 
$1_H' \equiv 1$ and $1 - 1_H \equiv 0$ because $H \cap H'$ has $n$-dimensional Hausdorff measure $0$. In particular, regressing $\cF$ on $L^2(\PP)$ or $L^2(\PP^H)$ reduces to estimating an element of $[0, 1]$ given $n$ noisy observations, for which the minimax rate is $\frac{1}{n}$.

On the other hand, let $\cP = \{\PP\} \cup \{\PP^H: H \in \Gr\}$, and consider the distribution-unaware minimax risk $\mathcal M^{(du)}_n(\cF, \cP)$. We claim that $\mathcal M^{(du)}_n(\cF, \cP) = \omega(1)$ \textit{even when there is no noise}. The intuition behind this is that when the estimator $\bar f_n$ observes $(X_1, Y_1), \ldots, (X_n, Y_n)$, it does not ``know'' whether the distribution is $\PP$ or $\PP^H$ where $H = \operatorname{Span}(X_1, \ldots, X_n)$, and thus it does not know whether to generalize the observations in a way which is consistent with the $L^2(\PP)$ norm or the $L^2(\PP^H)$ norm.

To see this formally, let $\bar f_n: \cD \to \cF$ be some estimator. For any $\vec X \in \cX^n$, let $\Span(\vec X) := \Span(X_1, \ldots, X_n)$, which is an element of $\Gr$ with probability $1$. In the noiseless setting, any sample $\vec X, \vec Y$ such that $\Span(\vec X) \in \Gr$ will have $\vec Y$ a multiple of $(1, \ldots, 1)$ with probability $1$, so we may as well consider $\vec Y$ to be a constant function. To get a lower bound on the minimax rate, it is also sufficient to consider only the extreme points of $\cF$, namely the functions in $\cF' = \{1_H, 1 - 1_H: H \in \Gr\}$, which are $\{0, 1\}$-functions, so we think of our estimators as functions $\bar f_n: \cX^n \times \{0, 1\} \to \cF$. We also assume for simplicity that $\bar f_n$ always returns a function in $\cF'$; if $\bar f_n$ is allowed to take values in the full convex hull $\cF$, one obtains the same lower bound on the risk of $\bar f_n$ by a more complicated version of the argument below.

Suppose, for example, that the estimator is given the sample $(\vec X, 1)$. In the noiseless setting, there is no point in returning a function inconsistent with the observations, so $\bar f_n$ must return either $1_H$ for $H = \Span(\vec X)$ or $1 - 1_{H'}$ for some $H' \neq H$ (not containing any of the $X_i$). Similarly, on the sample $(\vec X, 0)$ an optimal estimator $\bar f_n$ will return either $1 - 1_H$ or $1_{H'}$ for some $H' \neq H$. Let 
\begin{align*}
p_{H, 0} &= \Pr_{X_i \sim \PP^H} \{\bar f_n(\vec X, 0) = 1 - 1_H\} \\
p_{H, 1} &= \Pr_{X_i \sim \PP^H} \{\bar f_n(\vec X, 1) = 1_H \} \\
p_0 &= \Pr_{X_i \sim \PP} \{\bar f_n(\vec X, 1) = 1 - 1_{\Span(\vec X)}\} \\
p_1 &= \Pr_{X_i \sim \PP} \{\bar f_n(\vec X, 1) = 1_{\Span(\vec X)}\}.
\end{align*}

The Grassmannian $\Gr$, which is itself isomorphic to $\mathbb{S}^n$, has a uniform (i.e., rotationally invariant) probability measure, and one has $p_i = \mathbb E_{H \sim U(\Gr)}[p_{H, i}]$: choosing $n$ points from the unit sphere in $\mathbb R^{n + 1}$ is the same as choosing a uniform hyperplane and then choosing $n$ points uniformly from that hyperplane, by rotational invariance. 

One computes that $p_{H, i}$ and $p_i$ determine the error of $\bar f_n$ on $\cF'$ as follows:
\[
\varepsilon_{\ft, \PP^H}^2 := \mathbb E\int (\bar f_n- \ft)^2\,d\PP^H = 
\begin{cases} 
p_{H, 0} & \ft \in \{1 - 1_H\} \cup \{1_H'\}_{H' \neq H} \\
p_{H, 1} & \ft \in \{1_H\} \cup \{1 - 1_H'\}_{H' \neq H}
\end{cases},
\]
\[
\varepsilon_{\ft, \PP}^2 := \mathbb E\int (\bar f_n-\ft)^2\,d\PP = 
\begin{cases} 
1 - p_0 & \ft \in \{1_H\}_{H \in \Gr} \\
1 - p_1 & \ft \in \{1 - 1_H\}_{H \in \Gr}
\end{cases}.
\]
To lower-bound the distribution-unaware minimax risk, we consider the expected error of $\bar f_n$ under two different scenarios: when we choose a hyperplane $H$ uniformly at random and measure the error of $\bar f_n$ on the function $1_H$ when the input distribution is $\PP^H$, and when we fix $\ft = 1 - 1_{H_0}$ and measure the expected error of $\bar f_n$ when the distribution is $\PP$. By the above, the error in the first scenario is $\E_{H \sim U(\Gr)}[p_{H, 1}] = p_1$, while the error in the second scenario is $1 - p_1$. As $\max \{1 - p_1, p_1\} \ge \frac{1}{2}$, this shows that $\mathcal M^{(du)}_n(\cF, \cP) = \omega(1)$, as desired.


\begin{remark}\label{Rem:Measurability}
    Example \ref{Ex:One} may seem unnatural, as the measures $\PP$ and $\PP^H$ are all mutually singular, which leads to the ``collapse'' of the function class in different ways in $L^2(\PP)$ and each $L^2(\PP^H)$. To exclude such pathology, one might wish to consider only families of distributions all of which are absolutely continuous with respect to some reference measure $\PP^{(0)}$. It is not difficult, though, to modify Example \ref{Ex:One} in such a way which avoids any measurability issues: for given $n$, let $F = \mathbb F_q$ be the finite field of cardinality $q$ for some $q > n$, let $\cX = F^{n + 1}\backslash \{0\}$, let $\PP$ be the uniform probability measure on $\cX$, and for each $H$ in the set $\Gr(F)$ of $n$-dimensional linear subspaces of $F^{n + 1}$, let $\mathbb P^H$ be the uniform probability measure on $H \backslash \{0\} \subset \cX$, let $\cF = \conv \{1_H, 1 - 1_H: H \in \Gr(F)\}$, and let $\cP = \{\PP\} \cup \{\PP^H\}_{H \in \Gr(F)}$ as above. (Note that all measures in $\cP$ are absolutely continuous with respect to $\PP$.) As $q > n$, each hyperplane $H$ has $\PP$-measure $O(n^{-1})$ in $\cX$ and for any $H' \neq H$, $H \cap H'$ has $\PP^H$-measure $O(n^{-1})$, so one easily verifies all the computations in the example are still valid, up to errors of order $O(n^{-1})$.
\end{remark}
\end{example}

\end{document}